\documentclass[11pt]{amsart}

\usepackage{euscript}
\usepackage{amsmath}
\usepackage{amsthm}
\usepackage{epsfig}
\usepackage{amssymb}\usepackage{amscd}
\usepackage{epic}
\usepackage{eufrak}

\numberwithin{equation}{section}
\numberwithin{equation}{subsection}

\theoremstyle{plain}
\newtheorem{theorem}[equation]{Theorem}
\newtheorem{lemma}[equation]{Lemma}

\newtheorem{question}[equation]{Question}

\newtheorem{thm}[equation]{Theorem}
\newtheorem{cor}[equation]{Corollary}
\newtheorem{lem}[equation]{Lemma}
\newtheorem{prop}[equation]{Proposition}

\theoremstyle{definition}
\newtheorem{example}[equation]{Example}
\newtheorem{remark}[equation]{Remark}

\newtheorem{defn}[equation]{Definition}
\newtheorem{nota}[equation]{Notation}
\newtheorem{ex}[equation]{Example}

\newtheorem{rem}[equation]{Remark}

\numberwithin{equation}{section}
\numberwithin{equation}{subsection}

\newcommand{ \lk }{ \mbox{lk} }

\newcommand{ \imm }{ \mbox{Imm} }

\def\C{\mathbb C}

\def\R{\mathbb R}
\def\Z{\mathbb Z}

\newcommand{\labelpar}{\label}

\title{Immersions associated with holomorphic germs}

\author{Andr\'as N\'emethi}
\address{A. R\'enyi Institute of Mathematics, 1053 Budapest, Re\'altanoda u. 13-15, Hungary.}
\email{nemethi.andras@renyi.mta.hu}

\author{Gerg\H{o} Pint\'er}
\email{pinter.gergo@renyi.mta.hu}
\thanks{NA is partially supported by OTKA Grant 100796, PG is supported by `Lend\"ulet' and ERC program `LTDBud' at R\'enyi
Institute.}



\date{}

\begin{document}

\maketitle


\pagestyle{myheadings} \markboth{{\normalsize
A. N\'emethi and G. Pint\'er}}{ {\normalsize About the immersions associated with holomorphic germs}}

\begin{abstract}
 A holomorphic germ $ \Phi: ( \C^2, 0) \to ( \C^3, 0) $,  singular only at the origin,
  induces at the links level an immersion of $ S^3 $ into $ S^5 $.
The regular homotopy type of immersions $S^3\looparrowright S^5$ are determined by their Smale invariant,
defined up to a sign ambiguity.
In this paper we fix a sign of the  Smale invariant and we show that
for immersions induced by holomorphic gems the
sign--refined Smale invariant $\Omega$ is the negative of the
number of cross caps appearing in a generic perturbation of $ \Phi $.
Using the algebraic method we calculate $\Omega$
 for some families of singularities, among others the A-D-E quotient singularities. As a corollary,
  we obtain that the regular homotopy classes which admit holomorphic
  representatives are exactly    those, which have non-positive sign--refined Smale invariant.
  This answers a question of Mumford regarding exactly this correspondence.
We also determine the sign ambiguity in the
topological formulae of Hughes--Melvin and Ekholm--Sz\H{u}cs connecting the Smale invariant
with (singular) Seifert surfaces. In the case of holomorphic realizations of Seifert surfaces,
we also determine  their involved invariants  in terms of holomorhic geometry.
\end{abstract}

\section{Introduction}\labelpar{s:i}

\subsection{}\label{sec:1.1}
Let $S^n$ denote the $n$--sphere, the boundary of the unit ball in $\R^{n+1}$. The regular homotopy classes of immersions  $ f:S^3 \looparrowright S^5 $,
denoted by ${\rm Imm}(S^3,S^5)$,
are identified with the elements of $ \pi_3 (V_3( \R_5 )) \cong \pi_3 (SO(5)) \cong \Z $ by the Hirsch-Smale theory
 \cite{hirsch,smale}.
The correspondence is given by the \emph{Smale invariant}  $ \Omega(f) $ of an immersion $f$.
Besides the original definition of Smale \cite{smale},
there are several equivalent definitions  of $\Omega(f)$
(see \cite{hirsch,HM,szucstwo}). Usually, in all these constructions
there is no identification of a distinguished generator of $\pi_3(SO(5))$,
 hence the Smale invariant is well--defined only up to a sign.

The  subgroup ${\rm Emb}(S^3,S^5)$ of ${\rm Imm}(S^3, S^5)$ consists of the
regular homotopy classes which admit embedding representatives.
By  \cite{HM} this is the subgroup $24\cdot \Z\subset \Z={\rm Imm}(S^3,S^5)$.
For embeddings  the Smale invariant has the following alternative definition
too, given by Hughes and Melvin.
Let $M^4$ be a `Seifert surface'  in $S^5$ of
$f(S^3)$, then  $2\Omega(f)/3$ is the signature of $F$ (up to a sign), cf. \cite{HM}.
(This and similar identities will be reviewed in section~\ref{s:ss}.)

Our goal is to analyse the complex analytic realizations of the elements of the
above two groups.
Let $ \Phi: ( \C^2, 0) \to ( \C^3, 0) $ be a holomorphic germ. We assume that $ \Phi $ is singular only at the origin, that is  $ \{z\,:\, {\rm rank}(d \Phi_z)< 2\} \subset \{0\}$
in a small representative of $(\C^2,0)$. Such a germ, at the level of links of the spacegerms
$(\C^2,0)$ and $(\C^3,0)$, provides
an immersion $ f: S^3 \looparrowright S^5 $ (see \ref{ss:link}).
If an element of ${\rm Imm}(S^3,S^5)$, or ${\rm Emb}(S^3, S^5)$ respectively,
can be realized (up to regular homotopy)
by such an immersion,  we call it {\it holomorphic}. The corresponding subsets
will be denoted by ${\rm Imm}_{hol}(S^3,S^5)$ and  ${\rm Emb}_{hol}(S^3, S^5)$ respectively.

As we will see, ${\rm Imm}_{hol}(S^3, S^5)$ is not symmetric with respect to a sign change of $\Z$,
hence, in order to identify the subset ${\rm Imm}_{hol}(S^3,S^5)$ without any sign-ambiguity,
we will fix a `canonical' generator of  $ \pi_3 (SO(5)) $. This will be  done via the
ismorphisms  $ \pi_3 (U(3))\to \pi_3(SO(6))\to \pi_3(SO(5))$ and by fixing  a canonical generator in $ \pi_3 (U(3))$ (see \ref{ss:sign}).
Sometimes, to emphasize that we work with the Smale invariant with this fixed sign convention,
we refer to it as the {\it sign--refined Smale invariant}.
Our second goal is to determine the correct signs (compatibly with the above
choice of generators) in the existing topological formulas, which were stated only up to a
sign--ambiguity.

\subsection{The set ${\rm Imm}_{hol}(S^3,S^5)$}
One expects that the analytic geometry of  holomorphic
realization imposes some rigidity restrictions, and also provides some further connections with the
 properties of complex analytic spaces.
Mumford already in 1961 in his seminal article \cite{mumford} asked for the
characterization of the Smale invariant of a holomorphic (algebraic) immersion in terms of
the analytic/algebraic geometry. This article provides a complete answer to his question.
A more precise formulation of our guiding questions are:

\begin{question} \label{question}\

(a) Which are the regular homotopy classes ${\rm Imm}_{hol}(S^3,S^5)$ and  ${\rm Emb}_{hol}(S^3, S^5)$
 represented by holomorphic  germs?

(b) How can a certain  regular homotopy class be identified via complex singularity theory,
that is, via algebraic or analytic invariants of the involved analytic spaces?
Furthermore, if some $\Phi$ realizes some Smale invariant (e.g., if its Smale invariant is zero), then what kind of specific analytic properties $\Phi$  must have?
\end{question}
The main results  of this paper provide the following answer in the case of immersions.

\begin{thm}\label{th:main}

(a) ${\rm Imm}_{hol}(S^3,S^5)$ is identified via the sign--refined
Smale invariant $\Omega(f)$ by the set of non--positive integers.

(b) If the immersion $f$ is induced by the holomorphic germ $\Phi$,
then $ \Omega(f) = -C( \Phi)$, where  $C( \Phi) $ is
the number of cross cap points (complex Whitney umbrellas, or pinch points)
of a generic perturbation of $ \Phi $. $ C( \Phi ) $ can be calculated in an algebraic way,
as the codimension of the ideal generated by the determinants of the $2\times 2$-minors
of the Jacobian matrix of $\Phi$.
\end{thm}

The main tool of the proof of Theorem~\ref{th:main} is the concept of  \emph{complex Smale invariant}
of the germ $\Phi$. We introduce it in section~\ref{s:co} and then we prove that it
agrees with $ C( \Phi) $. Next,  in section~\ref{s:proof}
we identify the complex Smale invariant of a germ $ \Phi $ with the (classical) Smale invariant of the
 link of $ \Phi $. The proof of the part (b) of Theorem~\ref{th:main} is then ready up to sign. In \ref{ss:sign} we fix explicit generators of the groups $ \pi_3 (U) $ and $ \pi_3 (SO) $
 and calculate the homomorphism between them.
With this convention the complex Smale invariant of $ \Phi $ is equal to $ C( \Phi) $ and is opposite to
the sign--refined Smale invariant.

Part (b) of Theorem~\ref{th:main} implies that the sign--refined Smale invariant of a complex
analytic realization
is always non--positive. The proof of  part (a) is then completed by Example~\ref{ex:1}, which provides analytic  representatives for all non--positive $\Omega(f)$.

Note that in the present literature  the known ($C^\infty$)
 realizations of certain Smale invariants  $\Omega(f)$
 are rather involved (similarly, as the  computation of $\Omega(f)$ for any concrete $f$), see e.g. \cite{hughes, ekholm3}.
Here we provide very simple polynomial maps realizing all  non--positive Smale invariants.
Furthermore, the computation of $C(\Phi)$ for any $\Phi$ is extremely simple.

Moreover, precomposing the above complex realizations with the $C^\infty$ reflection
$ (s, t) \mapsto (s, \bar{t})$, we get explicit representatives for all positive Smale numbers
 as well, compare \cite[Lemma 3.4.2.]{ekholm3}.

\subsection{The set  ${\rm Emb}_{hol}(S^3,S^5)$.}
Recall that ${\rm Emb}_{hol}(S^3,S^5)$ consists of regular homotopy classes (that is,
sign--refined Smale invariants in $\Z$)
represented by holomorphic germs $\Phi$ whose induced  immersions $S^3\looparrowright S^5$
might not be embeddings, but are regular homotopic with embeddings.

A more restrictive subset consists of those regular homotopy classes (Smale invariants),
which can be represented by holomorphic gems, whose restrictions off origin are embeddings.

\begin{theorem}\label{th:embintro}
(a)  ${\rm Emb}_{hol}(S^3,S^5)= (24\cdot \Z)\cap\Z_{\leq 0}$.

(b) Assume that the immersion $f$ is the restriction at links level of a
holomorphic germ $\Phi$ as above, $f=\Phi|_{S^3}$. Then the following facts are equivalent:

\begin{enumerate}
\item ${\rm rank}\, d\Phi_0=2$ (hence $\Phi$ is not singular),
\item $\Omega (f)=0$,
\item $f:S^3\hookrightarrow S^5$ is an embedding,
\item $f: S^3 \hookrightarrow S^5 $ is the trivial embedding.
\end{enumerate}
\end{theorem}
Again, we wish to emphasize that the previous
construction of the generator of $24\cdot \Z={\rm Emb}(S^3,S^5)$ (that is, of a smooth embedding
with $\Omega(f)=\pm 24$) is complicated, it is more existential than constructive \cite{HM}.
On the other hand, by our complex realizations, for any given $\Omega(f)\in 24\cdot \Z$
we provide several easily defined germs, which are immersions, and  are regular homotopic with embeddings.
Moreover, part (b) says that it is impossible to find holomorphic representatives $\Phi$ such that
$\Phi|_{S^3}$ is already embedding (except for $\Omega(f)=0$).

The essential  parts of Theorem \ref{th:embintro}(b) are the implications
(2) $\Rightarrow$ (1) and  (3) $\Rightarrow$ (1), which conclude an analytic statement
from topological ones. The proof (2) $\Rightarrow$ (1) is based on Theorem \ref{th:main}, which
recovers the vanishing of the analytic invariant $C(\Phi)$ from the `topological vanishing' $\Omega(f)=0$.

A possible proof of
($\Phi|_{S^3}$ embedding) $\Rightarrow$ (${\rm rank}\, d\Phi_0=2$)
is based on a  deep theorem of Mumford,  which says that if the link of a complex normal surface
singularity is $S^3$ then the germ should be non--singular \cite{mumford}.
We will provide two other  possible proofs too: one of them is based on Mond's Theorem
\ref{th:C}, the other on a theorem of Ekholm--Sz\H{u}cs \cite{ESz} (a generalization
of the already mentioned Hughes--Melvin resuls \cite{HM}). These theorems will
be discussed in connection with properties of Seifert surfaces in section \ref{s:ss} as well.

\subsection{} The literature of singular analytic germs $\Phi:(\C^2,0)\to (\C^3,0)$ is huge with several deep and interesting results and invariants, see e.g. the articles of
D. Mond and V. Guryunov and the references therein. In singularity classifications
finitely determined or finite codimensional germs are central (with respect to some
equivalence relation). For germs $\Phi:(\C^2,0)\to (\C^3,0)$
Mond proved that the finiteness of the right-left codimension is equivalent with the finiteness
 of three invariants: the number of (virtual) cross caps $C(\Phi)$,
 the number of  (virtual) ordinary triple points $T(\Phi)$,
  and an other invariant $N(\Phi)$ measuring the non--transverse selfintersections  \cite{Mond2}.
This is more restrictive than our assumption $ \{z\,:\, {\rm rank}(d \Phi_z)< 2\}=\{0\}$,
which requires the finiteness of $C(\Phi)$ only.

However, it is advantageous   to consider this larger class, since there are
many key families of
germs  with infinite right-left codimension, but with finite $C(\Phi)$, and they
produce interesting connections with other areas as well (see e.g. the next example).

\begin{example}\label{ex:ade} Consider a simple hypersurface singularity $(X,0)\subset (\C^3,0)$
(that is, of type A--D--E). They are quotient singularities, that is $(X,0)\simeq(\C^2,0)/G$
for certain finite subgroup $G\subset GL(2,\C)$. Let $K$ be the link of $(X,0)$ (e.g.,
it is a lens space for A-type), and consider
the regular $G$--covering $S^3\to K$.
This composed with the inclusion $K\hookrightarrow S^5$ provides an {\it immersion}
$S^3\looparrowright S^5$. Hence, the universal cover of each  A--D--E singularity
automatically provides an element
of ${\rm Imm}_{hol}(S^3,S^5)$, which usually have infinite right--left codimension.
The corresponding Smale invariants are given in section~\ref{s:ex}. E.g.,
$-\Omega(A_{n-1})=n^2-1$, hence $A_4$ represents (up to regular homotopy) a generator of
$24\cdot \Z={\rm Emb}(S^3,S^5)$.

Recently Kinjo, using the plumbing graphs of the links of
A--D singularities and $C^\infty$--techniques, constructed
immersions with the same Smale invariants
as our $-C(\Phi)$ \cite{kinjo}. Hence, the natural complex analytic
maps $(\C^2,0)\to (X,0)\subset (\C^3,0)$  provide analytic realizations of the $C^{\infty}$
constructions of \cite{kinjo}, and emphasize their distinguished nature.
\end{example}

\subsection{Smale invariants and the geometry of Seifert surfaces.}
In section \ref{s:ss} we review three major topological  theorems,
which recover the classical Smale invariant in terms of the geometry of their
Seifert surfaces (namely the Hughes--Malvin theorem \cite{HM}, and two
 theorems of Ekholm--Sz\H{u}cs \cite{ESz}).
 All of them carry the sign ambiguity of the Smale invariant (which sometimes
is also caused by the nature of their proofs).

Section \ref{s:eszcomp} has two goals. First, we will indicate  the correct sign
in all these formulae, whenever the Smale invariant is replaced by the sign--refined
Smale invariant. Moreover, we also determine the Seifert type invariants
in terms of $C(\Phi)$ and $T(\Phi)$, whenever the immersion is induced by a holomorphic
germ $\Phi$.

When $f$ is a generic immersion, the invariant $L(f) $ of generic immersions introduced by Ekholm \cite{ekholm3} is also expressed in terms of $C(\Phi)$ and $T(\Phi)$.

\subsection{Acknowledgements.} The authors are very grateful to Andr\'as Sz\H{u}cs for
several very helpful conversations regarding different definitions and properties
of the  Smale invariant. Without his advises this works would not be completed.
We also thank Tam\'as Terpai for several discussions and advices
regarding topological invariants of
singular maps.

\section{Basic definitions and preliminary properties}\labelpar{s:bas}
\subsection{The immersion associated with $\Phi$}\labelpar{ss:link}
If $(X,0)$ is a complex analytic germ with an isolated singularity $0\in X$, then its link
 $K$ can be defined as follows. Set a real analytic map $\rho:X\to [0,\infty)$ such that
 $\rho^{-1}(0)=\{0\}$. Then, for $\epsilon>0$ sufficiently small, $K:=\rho ^{-1}(\epsilon)$
 is an oriented manifold, whose isotopy class (in $X\setminus \{0\}$)
 is independent of the choices, cf.
  Lemma (2.2) and Proposition (2.5) of \cite{looijenga}.
 E.g., if $(X,0)$ is a subset of $(\C^N,0)$, then one can take the restriction of $\rho(z)=|z|$ (the norm of $z$). In this way, the link of $(\C^N,0)$ is the sphere $S^{2N-1}_\epsilon$.
 Nevertheless, the general definition  is very convenient
 even if $(X,0)=(\C^N,0)$.

Let  $ \Phi: ( \C^2, 0)  \to (\C^3, 0) $ be a holomorphic germ singular only at $0$
(as in the introduction). Define $\rho:(\C^2,0)\to [0,\infty)$ by $\rho(z)=|\Phi(z)|$.
Since  $\Phi^{-1}(0)=\{0\}$ (if $\Phi^{-1}(0)$ would be a  positive dimensional
germ, then along it the rank of $d\Phi_z$ would be $<2$). Hence  $\rho^{-1}(0)=\{0\}$
(in a small representative).

\begin{lemma}\label{cor:epsilon}
There exists an $ \epsilon_0 > 0 $ sufficiently small such that
$\mathfrak{B}_\epsilon:=\Phi^{-1} ( \{z:|z|\leq \epsilon\} )$  is a
non-metric $C^\infty$ closed ball around the origin of\, $\C^2$. Its boundary,
$\Phi^{-1}(S^5_{\epsilon})$ is canonically diffeomorphic to $ S^3$ for
any $ \epsilon < \epsilon_0 $. In fact, for $\tilde{\epsilon}$ with
$0<\tilde{\epsilon}\ll \epsilon$, any standard metric sphere
$S^3_{\tilde{\epsilon}}$ sits in $\mathfrak{B}_\epsilon$, and it is isotopic with
$ \Phi^{-1} (  S^5_{ \epsilon} ) $ in  $\mathfrak{B}_\epsilon\setminus 0$.
\end{lemma}
In the sequel $ \Phi^{-1} (  S^5_{ \epsilon } ) $ and $S^3_{\tilde{\epsilon}}\subset \C^2$
will be identified. When it is important to differentiate them we will use the notation
$ \mathfrak{S}^3:=  \Phi^{-1} (  S^5_{ \epsilon} ) $.
We write also $S^3=S^3_{\tilde{\epsilon}}$ and $S^5=S^5_\epsilon$.

\begin{defn}
We call the restriction  $ f= \Phi |_{ \mathfrak{S}^3 } : \mathfrak{S}^3 \looparrowright S^5 $ the immersion associated with  $ \Phi $. (It's regular homotopy class is independent of all the choices.)
\end{defn}


\subsection{The number of cross caps.}\labelpar{ss:Wh}
Let $ \Phi: ( \C^{2}, 0)  \to (\C^{3}, 0) $ be a holomorphic germ singular only at $0$.
Consider a generic holomorphic deformation
$ \Phi_{ \lambda} $ of $ \Phi = \Phi_0 $. The singular points of $ \Phi_{ \lambda \neq 0 } $ are
cross caps (or complex Whitney umbrellas), i.e. they have the local form $ ( s, t) \mapsto ( s^2, st, t) $
in some local holomorphic coordinates.
Their  number  does not  depend on the deformation $ \Phi_{\lambda} $,
it is an invariant of $ \Phi $, cf. \cite{Mond1,Mond2}.
We denote it by $ C ( \Phi ) $.

 $ C ( \Phi ) $ can be computed in an algebraic way as well.
 Let $ M_j: Hom ( \C^2 , \C^{3} ) \to \C $
 denote the determinants of the three $ 2 \times 2 $ minors ($ j=1, 2, 3 $).
 Let $ J $ be the ideal of the local ring $ \mathcal{O}_{ \C^2, 0} $
 generated by the elements $ M_j\circ d\Phi$, where $d\Phi$ is the complex Jacobian
 matrix.
 $J$ has finite codimension exactly when $\Phi$ is immersion off the origin.
\begin{thm}\label{th:C}\cite[Proposition 1]{Mond1}
$ C( \Phi ) = \dim_\C \, (\mathcal{O}_{ \C^2, 0}/J)$.
\end{thm}

\subsection{The number of triple points.}\label{ss:trip}
If $\Phi_\lambda$ is a generic deformation as above, then the singular points
of the {\it image} of $ \Phi_{\lambda \neq 0} $
 might have the following types:  self-transversal double points, cross caps and triple points, cf. \cite{Mond1,Mond2}. The double point set has complex dimension $1$, while triple points are
  isolated. 
If the codimension of the second fitting ideal of $ \Phi $ is finite, say $T(\Phi)$, then
the number of triple points $ T( \Phi_{\lambda \neq 0} ) $ of $ \Phi_{\lambda \neq 0} $
is independent of the deformation and $\lambda$, it is exactly $ T(\Phi)$.

\subsection{The Smale-invariant}\labelpar{ss:Smale} Let $ f : S^3 \looparrowright \R^5 $ be an immersion. Instead of the original definition of Smale \cite{smale} we adopt the
construction of Hughes and Melvin for  the \emph{Smale invariant} of $ f $, see \cite{HM}, compare also with \cite{szucstwo}.
Let $ U $ be a tubular neighborhood of the standard $ S^3 \subset \R^5 $, and let $ F: U \looparrowright \R^5 $ be an orientation preserving immersion extending $f $, i.e. $ F|_{S^3} = f $. Let $ TU $
be the tangent bundle of $U$. It inherits a global trivialization from
the natural trivialization of $T\R^5 $. In particular,  there is a map (the Jacobian matrix)
\[ dF|_U : U \to GL^+ (5, \R ).
\]
Its homotopy class is the Smale invariant of $f$:
\begin{equation}\label{eq:Smale}
\Omega(f) = [ dF|_{S^3} ] \in \pi_3 (SO(5))
\end{equation}
(via the  homotopy equivalence induced by the inclusion $ SO(5) \subset GL^+ (5, \R ) $).

\begin{remark}\label{r:lie}
 If $ G $ is a connected Lie group, or a factor of it by a
 closed connected subgroup, then $ \pi_n (G) $ can be identified with the homotopy classes
 of the continuous maps $ f: S^n \to G $ without any base point. Furthermore, for Lie
 groups, the group operation of $\pi_n(G)$ agrees with that induced by
 the pointwise multiplication in $ G$; cf. \cite[p. 88 and 89]{steenrod}.
\end{remark}

\begin{prop}\labelpar{prop:Smale} $ \Omega(f) $ does not depend on the choice of
$U$ and $ F $, it
depends only on the regular homotopy class of $ f $, and
$ \Omega: {\rm Imm} (S^3, \R^{5} ) \to \pi_3 (SO(5) ) $ is a bijection.

\end{prop}
Indeed, Smale proved that his original invariant gives a bijection between
$ \imm (S^3, \R^{5} )$ and $ \pi_3 (V_3 ( \R^{5} )) $, cf. \cite{smale},
where $V_3(\R^5)$ denotes  the real Stiefel manifold (the space of
 linear independent $3$-frames of $ \R^5$).
Hughes and Melvin proved that their alternative definition (\ref{eq:Smale})
of the Smale invariant does not depend on the choice of $ F $ and agrees with the original Smale invariant through the natural
group isomorphism $ \pi_3 (SO(5)) \to \pi_3 (V_3 ( \R^{5} )) $.

Note that the standard embedding $ SO(5) \hookrightarrow SO $ induces a group isomorphism between $ \pi_3 (SO(5)) $ and $ \pi_3 (SO) $. These groups are (a priori non--canonically) isomorphic to $ \Z $.

\vspace{2mm}

We wish to emphasize the following facts regarding orientations of $S^3$ and $\R^5$ 
 and their effects on the above definition. (This might serve also as a small guide for the next sections.)
  
Let us think about $S^3$ as the subset of $\R^4$, the boundary of the 4--ball $B^4$ in $\R^4$,
or via embedding $\R^4\subset \R^5$, as a subset of $\R^5$. We do not wish to fix any orientation on it 
as the orientation of $\partial B^4$
(that would depend on the convention how one defines the orientation of the boundary of an oriented
manifold --- called, say, `boundary convention'). 

Note that in the above definition
of the Smale invariant, not the orientation of $S^3$ is used, but the orientation
of the tubular neighborhood $U\subset \R^5$ and the orientation of the target $\R^5$. Moreover, 
$\Omega(f)$ is unsensitive to the orientation change simultaneously in both $\R^5$.
In this way we get an element $\Omega(f)\in[S^3,SO(5)]$, which is independent of the orientation of 
$\R^5$ and does not use any orientation of $S^3$. Furthermore, if we define 
(this will done in \ref{ss:sign}) a generator $[L]$ in $[S^3,SO(5)]$, using again only
the embedding $S^3\subset \R^5$ (and no other orientation data), then $\Omega(f)$ identifies with
an element of $\Z$, such that its definition is independent of any orientations of $S^3$ and 
$\R^5$, hence also of the `boundary convention'. 

All our discussions are in this spirit (except sections \ref{s:ss} and \ref{s:eszcomp}, where
oriented Seifert surfaces are treated): we run orientation and `boundary convention' free
definitions and statements (associated with  $S^3$, regarded as a subset of  $\R^5$, and immersions 
$S^3\looparrowright \R^5$).

 However, if we fix a `boundary convention', then $S^3$ (in $\R^5$) will get an orientation (as $\partial 
 B^4$). Then, for any {\it oriented abstract $S^3$}, let us denote it by ${\bf S}^3$, and immersion
 ${\bf S}^3\looparrowright \R^5$, we can define the Smale invariant $\Omega^a(f)\in \Z$ 
 (here `a' refers to the `abstract' ${\bf S}^3$) by identifying ${\bf S}^3$ with the embedded 
 $S^3\subset \R^5$ by an orientation preserving diffeomorphism and taking  $\Omega(
 S^3\to {\bf S}^3\looparrowright \R^5)$. This $\Omega^a(f)$
 depends on the `boundary convention', since the identification 
 $S^3\to {\bf S}^3$ depends on it: changing the convention we change $\Omega^a(f)$ by a sign. 
 
 This point of view should be adapted when ${\bf S}^3$ will be the (oriented) boundary of an oriented 
 Seifert surface. But till section \ref{s:ss} we will focus on the first version, 
 $\Omega(f)$.
 
 \vspace{2mm}

Next, in the definition of $\Omega(f)$,
one can replace $\R^5$ by $S^5$, where $S^5$ is the boundary of the ball in $\R^6$, and
$S^3$ is embedded naturally in $S^5$.
By taking a generic
point $P\in S^5$ we identify $S^5\setminus \{P\}$ with $\R^5$, and 
$U$ will be replaced by a tubular neighborhood of  $S^3$ in $S^5$. Then the previous definition of $\Omega(f)$ can be repeated for any immersion $S^3\looparrowright S^5$ (where
$S^3\subset S^5$) providing an element $[S^3,SO(5)]$, which becomes an integer once  
 a generator $[L]$ is constructed from the embedding $S^3\subset S^5$. Again, this Smale invariant 
$\Omega(f)$ will be independent of the orientations of $S^3$ and $S^5$, hence of the 
`boundary convention' as well.  

For immersions defined in subsection \ref{ss:link}, $\mathfrak{S}^3$ evidently sits naturally in 
$\C^2=\R^4$ (hence also in a certain $\mathfrak{S}^5=S^5\subset \C^3=\R^6$, cf. \ref{ss:6.1}).
This together with definition
(\ref{eq:Smale})  provide $\Omega(f)$ (which becomes an integer once $[L]$ will be constructed
in \ref{ss:sign}).

\section{The complex Smale invariant}\labelpar{s:co}
\subsection{} In this section we define
 the \emph{complex Smale invariant} $ \Omega_{\C} ( \Phi ) $
for a holomorphic germ $ \Phi: ( \C^2, 0)  \to (\C^3, 0) $,
 singular only at $ 0 $.  It will be the bridge between $ C ( \Phi ) $ and $\Omega(f)$.

\begin{defn}
Consider the map (with target the complex Stiefel variety $V_2(\C^3)$):
\[ d \Phi |_{ S^3} : S^{3} \to V_2 ( \C^3) \]
defined via the natural trivialization of the complex tangent bundles $T\C^2$
and $T\C^3$.
By definition, the complex Smale invariant of $ \Phi $ is the  homotopy class:
\[ \Omega_{\C} ( \Phi ) = [ d \Phi |_{ S^{3}}  ] \in \pi_{3} ( V_2 ( \C^{3}) ). \]
\end{defn}
\noindent
By the connectivity of the group of local coordinate transformations, $\Omega_\C(\Phi)$
is independent of the choice of local coordinates in $(\C^2,0)$ and $(\C^3,0)$.

\begin{rem}\labelpar{re:stab}
The projection $ U(3) \to  V_2 ( \C^{3}) $ induces an isomorphism between
$ \pi_{3} ( V_2 ( \C^{3})) $ and $ \pi_{3} ( U(3)) = \pi_{3} ( U) $. Hence,
 if we choose a normal vector field $ N_{ \Phi } $ of $ \Phi $, then the map
\[ ( d \Phi , N_{ \Phi } )|_{S^{3}} : S^{3} \to GL(3, \C)
\]
represents $ \Omega_{ \C} ( \Phi ) $ in $ \pi_{3} ( GL(3, \C)) = \pi_{3} (U(3)) = \pi_{3} (U) $.
A canonical choice of $ N_{ \Phi } $ could be the complex conjugate of the cross product of the partial derivatives of $ \Phi $:
\[ N_{ \Phi } (s, t) = \overline{  \partial_s \Phi (s, t) \times \partial_t \Phi (s, t) } \mbox{ .} \]
\end{rem}

\begin{rem}
$ \pi_3 (U) \cong \Z $ and in \ref{ss:sign}  we
 identify them through a fixed isomorphism.
 In this way $ \Omega_{ \C} ( \Phi ) $ becomes  a well-defined integer without any sign--ambiguity.
\end{rem}

\section{Distinguished generators and sign conventions}

\subsection{}\labelpar{ss:comprel}
There is a natural map $ \tau: U(3) \to SO(6) $, which replaces any entry
$ M_{ij} =  a+ bi $ of a matrix $ M \in U(3) $ by the  real $ 2 \times 2 $--matrix
$\begin{pmatrix}a&-b\\b&a\end{pmatrix}$.
A map $ F: \C^3 \to \C^3 $ can be regarded as a map $ \tilde{F}: \R^6 \to \R^6 $: if we
denote by $ z_j = x_j + i y_j $  ($ j= 1, 2, 3 $) the coordinates of $ \C^3 $, then for the components of $ F $ and $ \tilde{F} $ one has
\[ F_j(z_1, z_2, z_3) = \tilde{F}_{2j-1} (x_1, y_1, x_2, y_2, x_3, y_3) + i \tilde{F}_{2j} (x_1, y_1, x_2, y_2, x_3, y_3). \]
 Then $ \tau ( d_{\C} F ) = d_{\R} \tilde{F} $ holds for the complex Jacobian of $F$ and the real Jacobian of $\tilde{F}$.

Let $ j: SO(5) \hookrightarrow SO(6) $ denote the inclusion. It is well--known (see e.g.
\cite{husemoller}) that
\begin{equation}\label{eq:pi3}
\pi_3 (j): \pi_3 (SO(5)) \to \pi_3 (SO(6)) \ \ \mbox{is an isomorphism}.
\end{equation}
\begin{lemma}\label{pr:hom}
The homomorphism $ \pi_3 ( \tau ): \pi_3 (U(3)) \to \pi_3(SO(6)) $ is an isomorphism too.
\end{lemma}
\begin{proof} First, we provide a more conceptual  proof, which does not identify
distinguished generators.
Both sides are in the stable range (see \cite{husemoller}), hence
 we can switch to the homomorphism $ \pi_3 (U) \to \pi_3(O) $ induced by the embedding $ \tau: U \hookrightarrow O $. By (a proof of)
 Bott periodicity,
 the factor $ O/U $ is homotopically equivalent to the loopspace $ \Omega O $ of $ O$, cf. \cite{bottper}.
 Hence $ \pi_i (O/U) = \pi_i ( \Omega O ) = \pi_{i+1} (O) = 0 $ for $i=3$ and 4.
  Then the  isomorphism follows from
 the homotopy exact sequence of the fibration $ O \to O/U $ with fibre $ U $.
\end{proof}
In \ref{ss:sign} we will give another, more computational  proof, where we will
be able to fix distinguished  generators for $ \pi_3 (U(3)) $ and $ \pi_3(SO(6)) $,
 and via these  generators  we
identify $ \pi_3 ( \tau ) $ with multiplication by $ -1 $.

\subsection{Conventions, identifications}\labelpar{ss:sign}
First, we identify $ \mathbb{H} $ and $ \R^4 $ and $ \C^2 $ in the obvious way:
 we identify the quaternion $ q = a + bi + cj + dk = z + wj  \in \mathbb{H} $ with $ (a, b, c, d) \in \R^4 $ and with the complex pair $ (z, w) \in \C^2 $, where $ z=a+bi $ and $ w= c+ di $. Also, we identify $ S^3 $ with the quaternions of unit length:
$ S^3 = \{ q= a + bi + cj + dk  \in \mathbb{H} \ | \ a^2 + b^2 + c^2 + d^2 = 1 \} $.

\begin{nota}
We define the following maps. Set
\[ u: S^3 \to U(2) \mbox{ , } u_q =
\left(
\begin{array}{cc}
z & - \bar{w} \\
w & \bar{z} \\
\end{array}
\right),
\]
where $ q= z + wj $. $ u_q $ is the (complex) matrix of the \emph{right}
(quaternionic) multiplication
with $ q $, that is, of the map $ \mathbb{H} \to \mathbb{H} $, $ p \mapsto pq $.
Note that the left multiplication by $ q $ is not a complex unitary transformation, in general.
Next, set
\[ L: S^3 \to SO(4) \mbox{ , } L_q =
\left(
\begin{array}{cccc}
a & -b & -c & -d \\
b & a & -d & c \\
c & d & a & -b \\
d & -c & b & a \\
\end{array}
\right),
\]
where $ q= a + bi + cj + dk $.
$ L_q $ is the (real) matrix of the \emph{left} multiplication with $ q $ (i.e., of
the map $ \mathbb{H} \to \mathbb{H} $, $ p \mapsto qp $).

Let $ R:  S^3 \to SO(4) $ be the map which assigns for a $q \in S^3 $ the (real) matrix $ R_q $ of the righ multiplication with $ q$ (i.e., of the map $ \mathbb{H} \to \mathbb{H} $, $ p \mapsto pq $).

Let $ \rho:  S^3 \to SO(4) $ be the map which assigns for a $q \in S^3 $ the (real) matrix $ \rho_q $ of the conjugation with $ q$ (i.e., of the map $ \mathbb{H} \to \mathbb{H} $, $ p \mapsto qpq^{-1} $).

We use the same notation for the compositions of these maps with the inclusions $ SO(4) \hookrightarrow SO $ and $ U(2) \hookrightarrow U $. Note that these inclusions commute with $ \tau $.

\end{nota}

\begin{prop}\cite[section 7, subsection 12]{husemoller}

 (a) $ \pi_3 (U(2)) = \pi_3 (U) = \Z \langle [u] \rangle $.

 (b) $ \pi_3 (SO(4)) = \Z \langle [L] \rangle \oplus \Z \langle [\rho] \rangle $.

 (c) $ \pi_3 (SO) = \Z \langle [L] \rangle $ and $ [ \rho] = 2 [L] $ in $ \pi_3 (SO) $.
 \end{prop}
 In the sequel, {\bf using these base choices $[u]$ and $[L]$ we identify
 the groups $\pi_3(U)$ and $\pi_3(SO)$ with $\Z$.}
 Now we can state the explicit version of Proposition~\ref{pr:hom}.
\begin{prop}\label{pr:homo}
$ \pi_3 (\tau) ([u]) = - [L] \in \pi_3 (SO) $ holds for $ [u] \in \pi_3 (U) $.
\end{prop}
 \begin{proof}
From definitions $ \tau \circ u = R $ and $ \rho R= L $, thus $ \pi_3 (\tau) ([ u ] ) = [ R]  = [L]-[ \rho ] = -[L] $.
\end{proof}

\begin{remark}
 Let $ p: U(2) \to S^3 $ be the projection (choosing the first or the second column of
 the matrix). Then $ [u] \in \pi_3 (U(2)) $ is the unique generator for which
 $ \deg (u \circ p )=1 $.
\end{remark}

\begin{prop}\label{prop:u}
 Let $ \Phi : \C^2 \to \C^3 $ be the cross cap, i.e. $ \Phi (s, t) = (s^2, st, t) $. Then $ \Omega_{\C} ( \Phi ) = [u] $.
\end{prop}

\begin{proof}
The map $ d \Phi|_{S^3}: S^3 \to V_2 ( \C^3 ) $ represents $ \Omega_{\C} ( \Phi ) $.
We should compose this with  $ V_2 (\C^3 ) \to U(3) $, then
with the inverse of the inclusion $ U(2) \to U(3) $,  and finally with the projection $ U(2) \to S^3 $;
 and then
calculate the degree of the resulting map $ S^3 \to S^3 $. In fact,
along these compositions
we will use (the homotopically equivalent groups)
$ GL(2, \C ) $ and $ GL(3, \C ) $ instead of $ U(2)  $ and $ U(3) $.
Therefore, we will arrive in
$ \C^2 \setminus \{ 0 \} $ instead of $ S^3 $.
\[ d \Phi|_{S^3} : S^3 \to V_2 (\C^3) \mbox{ , }
(s, t) \mapsto
 \left( \begin{array}{cc}
 2s & 0 \\
 t & s \\
 0 & 1 \\
\end{array} \right).
\]
The first composition gives the map
\[
 S^3 \to GL(3, \C ) \mbox{ , }
 (s, t) \mapsto
 \left( \begin{array}{ccc}
 2s & 0 & N_1 \\
 t & s & N_2 \\
 0 & 1 & N_3 \\
\end{array} \right),
\]
where $ N_1 = \bar{t} $, $ N_2 = -2 \bar{s} $ and $N_3 = 2 \bar{s}^2 $ are the coordinates of the normal vector $ N_{\Phi} $ (see Remark~\ref{re:stab}).
This is modified by the homotopy
\[
 S^3 \times [0,1] \to GL(3, \C ) \mbox{ , }
 (s, t, h) \mapsto
 \left( \begin{array}{ccc}
 2s & 0 & N_1 \\
 t & h s & N_2 \\
 0 & 1 & h N_3 \\
\end{array} \right),
\]
which maps $ (s, t, 0) $ into $ GL(2, \C ) \subset GL(3, \C) $. [Note that the determinant
 is $ |t|^2 + 4|s|^2 + 4h^2 |s|^4 \neq 0 $, thus the image is indeed in $ GL(3, \C) $.]
Hence, we obtain  the map
\[
 S^3 \to GL(2, \C ) \mbox{ , }
 (s, t) \mapsto
 \left( \begin{array}{cc}
 2s & N_1 \\
 t & N_2 \\
\end{array} \right),
\]
which composes with the projection (first column) provides
$ S^3 \to \C^2 \setminus \{0 \}$, $(s, t) \mapsto (2s, t)$.
After a normalisation, the degree of the resulting map is $ 1 $.
\end{proof}

\begin{rem}
The proof of Proposition~\ref{prop:u} works for all germs of the form
 \begin{equation}\labelpar{eq:rank}
  \Phi (s, t) = ( g_1 (s, t), g_2 (s, t), t )
 \end{equation}
and implies that
 \[
  \Omega_{\C} ( \Phi ) = \deg \left( S^3 \to S^3 \mbox{ , } (s, t) \mapsto \frac{(\partial_s g_1 (s, t), \partial_s g_2 (s, t))}{|(\partial_s g_1(s, t), \partial_s g_2 (s, t))|} \right).
 \]
 This degree agrees with the intersection multiplicity in $(\C^2,0)$
 of $ \partial_s g_1 $ and $ \partial_s g_2 $, i.e.
 \begin{equation}\labelpar{eq:other}
  \Omega_{\C} ( \Phi ) = \dim_{\C}\,  \frac{ \mathcal{O}_{\C^2, 0 }}{ (
  \partial_s g_1, \partial_s g_2 ) }.
 \end{equation}
 This also equals  $ C( \Phi) $ by Theorem~\ref{th:C}. This proves
 $\Omega_{\C} ( \Phi )= C(\Phi)$ for maps of corank 1.
  (All germs which satisfy $ {\rm rank} (d \Phi_{0} )= 1 $ are right--left equivalent with germs of type  (\ref{eq:rank})).
  This identity will be proved in the general case in section~\ref{sec:5}.

 Conversely, the identity $\Omega_{\C} ( \Phi )= C(\Phi)$ is proved in section~\ref{sec:5}
 independently of Theorem~\ref{th:C}, therefore (\ref{eq:other}) together with Theorem~\ref{th:Comp}
  give a new proof for Mond's Theorem~\ref{th:C}
  in the case of germs which satisfy $ {\rm rank} (d \Phi_{0} )= 1 $.
\end{rem}

\begin{rem}
The conventions we use are not universal. For example, Kirby and Melvin in \cite{KirbyMelvin}
 have chosen the same generators of $  \pi_3 (U) $ and $  \pi_3 (SO)$ (these are $ [u] $ and $ [L]$ with our notations), but they identified $ \R^4 $ and $ \C^2 $ differently than us. Namely, they identified the quaternion $ q= a+bi + cj + dk = z + jw \in \mathbb{H} $ with $ (a, b, c, d) \in \R^4 $ and the complex pair $ (z,w) \in \C^2 $, where $ z=a+bi $ and $ w= c-di $. With that identification $ u_q $ becomes the (complex) matrix of the quaternionic \emph{left} multiplication with $ q $. In that identification $ \pi_3 ( \tau )[u] $ would be equal to $ -[R] $, since that is the homotopy class of the map $ S^3 \to SO(4) $ given by the composition $ \tau \circ u \circ \kappa  $, where $ \kappa $ is the reflection $ \kappa (z, w) = (z, \bar{w}) $.
\end{rem}

\section{The identity $ \Omega_{ \C } ( \Phi ) = C ( \Phi ) $.}\label{sec:5}

\subsection{} Next we identify the complex Smale invariant with the number of cross caps.
\begin{thm}\label{th:Comp}
$ \Omega_{ \C } ( \Phi ) = C ( \Phi ) $.
\end{thm}

\begin{proof} Consider the following diagram:
\[
\begin{array}{cccc}
d \Phi: & \C^2 & \longrightarrow & Hom ( \C^2 , \C^3 ) \\
 & \cup & & \cup \\
d \Phi |_{ \C^2 \setminus \{ 0 \}} : & \C^2 \setminus \{ 0 \} & \longrightarrow &  Hom ( \C^2 , \C^3 ) \setminus \mathcal{D}= V_2 ( \C^3 )\\
\end{array} \mbox{ ,}
\]
where $ \mathcal{D} = \{ M \in Hom ( \C^2 , \C^3 ) \ | \ {\rm rank} (M) < 2 \} $.
$ \mathcal{D} $ is an irreducible algebraic variety of complex codimension $2$, its Zariski open set
$ \mathcal{D}^1 = \{ M \in {\mathcal D} \ | \ {\rm rank} (M) =1 \} $ is smooth.

First we prove that $ \Omega_{ \C } ( \Phi ) $ is equal to the linking number of
 $ d \Phi |_{S^3} $ and $ \mathcal{D} $ in $  Hom ( \C^2 , \C^3 ) $.
This is  defined as follows.
If  $ g: S^3 \to Hom ( \C^2 , \C^3 ) \setminus \mathcal{D} $ is a smooth map, and $\tilde{g}$
is a smooth extension defined on the ball such that $\tilde{g}|_{S^3} = g $ and
$ \tilde{g} $ intersects $ \mathcal{D} $ transversally along $ \mathcal{D}^1 $, then
the linking number of $ g $ and $ \mathcal{D} $
 is the algebraic number of the intersection points of $ \tilde{g} $ and $ \mathcal{D} $.
By standard argument it is a homotopy invariant of maps
$S^3 \to Hom ( \C^2 , \C^3 ) \setminus \mathcal{D} $.

The linking number gives a group homomorphism
$\lk: \pi_3 (V_2 ( \C^3 )) \to \Z $.
Next lemma shows  that this homomorphism is surjective.
\begin{lem}\label{lem:cr} Let $ \Phi(s, t) = (s^2, st, t) $ be the cross cap. If $ g= d \Phi|_{S^3} $ and $ \tilde{g} = d \Phi|_{B^4} $, then
$ \tilde{g} (0) \in \mathcal{D} $ is the only intersection point and the intersection is transversal at that point.
\end{lem}
\begin{proof} This is a straightforward local computation left to the reader.
The transversality follows also from the conceptual fact that the cross cap is a stable map.
\end{proof}

The sign of the intersection multiplicity at the intersection point of two
complex submanifolds is always positive. For $ g $ described in \ref{lem:cr}
 the linking number of $ g(S^3) $ and $ \mathcal{D} $ is $1$.
 This shows not only that the
 homomorphism given by the linking number is surjective (hence an isomorphism too), but also that
 this isomorphism agrees with the chosen one in \ref{ss:sign}.
  This follows from the fact that the complex Smale invariant of the cross cap is
   exactly the chosen generator, see Proposition~\ref{prop:u}. Hence, the homomorphisms  $\Omega_\C$
   and $\lk$ coincide.

 Next, we show that
 $ \lk_{ Hom ( \C^2 , \C^3 ) } ( d \Phi |_{S^3}(S^3), \mathcal{D}) = C ( \Phi ) $.
 Take a generic perturbation $ \Phi_{ \epsilon} $ of $ \Phi $. $ d \Phi_{ \epsilon} |_{S^3} $ is homotopic to $  d \Phi |_{S^3} $, hence  their linking numbers are the same. $ \Phi_{ \epsilon} $ has only cross cap singularities, their number is $ C ( \Phi ) $. This means that $ d \Phi_{ \epsilon} |_{B^4} $ intersects transversally $ \mathcal{D} $ in $ C ( \Phi ) $ points.
 Intersection of  complex manifolds provides positive signs.
\end{proof}

\begin{cor}
$ \Omega_{ \C } ( \Phi ) \geq 0 $.
\end{cor}

\section{The proof of Theorems~\ref{th:main} and \ref{th:embintro}}\labelpar{s:proof}
\subsection{}\label{ss:6.1}

 Theorem \ref{th:main} follows from Theorem \ref{th:Comp} and the next identity.
\begin{prop}
$ \pi_3 ( \tau ) ( \Omega_{ \C } ( \Phi ) ) = \pi_3 (j) (\Omega(f)) $.
\end{prop}

\begin{proof}
By the definition of the Smale invariant, one has to extend $ f $ to a neigbourhood of
the standard embedding of $ \mathfrak{S}^3 $ in an $ \R^5 $  (cf. \ref{ss:Smale}).
On the other hand $ \Phi $ extends $ f $ in the $ \C^2 $ direction. We will compare these two
extensions using a  common extension $ F: W \to \C^3 $, where $ W $ is a suitable neighborhood of
$ \mathfrak{S}^3 $ in $ \C^3 $.

Let us consider a fixed $\epsilon$ which satisfies the properties of
 Corollary \ref{cor:epsilon}. We also write
$B^6_\epsilon =\{z:|z|\leq \epsilon\}\subset \C^3$, $S^5_\epsilon=\partial B^6_\epsilon$,
${\mathfrak B}^4_\epsilon:=\Phi^{-1}(B^6_\epsilon)$ for the $C^\infty$ ball in $\C^2$, and  $\mathfrak{S}^3_\epsilon:= \Phi^{-1}(S^5_\epsilon)$ for its boundary. (Late we will drop some
of the $\epsilon$'s.)
For positive  $\epsilon_1$, $\epsilon_2$ sufficiently closed to $\epsilon$, $\epsilon_1<\epsilon<\epsilon_2$, and for $0 <\rho\ll\epsilon$ one defines
$ F(s, t, r) = \Phi (s, t) + r \cdot  N_{ \Phi } (s, t) $, where $(s,t,r)\in W:=\Phi^{-1}(z:\epsilon_1<|z|<\epsilon_2)\times D^2_\rho$, $D^2_\rho$ is
the $\rho$-disc in $\C$, and $  N_{ \Phi} $ is the complex normal vector of $ \Phi $, see Remark~\ref{re:stab} .
Since the normal bundle of $f$ in $S^5_\epsilon$
is trivial (and since the transversality is an open property), we get that $F^{-1}(S^5_\epsilon) $ is diffeomorphic to $\mathfrak{S}^3_\epsilon\times D^2_\rho$.
In fact, if $p:\C^2\times D^2_\rho\to D^2_\rho$ is the natural  projection,
then for any $r\in D^2_\rho$ we can define  $\mathfrak{S}^3_{\epsilon,r}:=
F^{-1}(S^5_\epsilon)\cap p^{-1}(r)$. Then each  $\mathfrak{S}^3_{\epsilon,r}$ is a $C^\infty$ 3-sphere, being the boundary of the $C^\infty$ 4-ball $\mathfrak{B}^4_{\epsilon,r}
\subset p^{-1}(r)$. Then $F^{-1}(S^5_\epsilon)=\cup_{r\in D^2_\rho} \mathfrak{S}^3_{\epsilon,r}$. Moreover, $\mathfrak{B}^6:=\cup_{r\in D^2_\rho} \mathfrak{B}^4_{\epsilon,r}\subset \C^2\times \C$ is a thickened tubular neighbourhood of
$\mathfrak{B}^4_\epsilon\subset \C^2\times 0$, homeomorphic to the real 6-ball.
Its corners can be smoothed, hence we think about it as a $C^\infty$ ball. Its boundary $\mathfrak{S}^5:=\partial\mathfrak{B}^6$ (diffeomorphic to the 5-sphere)
is the union of $F^{-1}(S^5_\epsilon)$ (diffeomorphic to
$S^3\times D^2$) and $\cup_{r\in \partial D^2_\rho} \mathfrak{B}^4_{\epsilon,r}$
(diffeomorphic to  $B^4\times S^1$).

 In a point $ (s, t, 0) \in \mathfrak{S}^3_\epsilon\times \{0\} $
 the differential of $ F $ is
\[ dF (s, t, 0)  = ( \partial_s F(s, t, 0) , \partial_t F(s, t, 0) , \partial_r F(s, t, 0) )
= ( \partial_s \Phi (s, t) , \partial_t \Phi (s, t) , N_{ \Phi } (s, t) ) \mbox{ .}
\]
Thus,  the homotopy class of $ dF|_{\mathfrak{S}_\epsilon^3} $ equals $ \Omega_{ \C } ( \Phi ) $  (cf.  \ref{re:stab}). Therefore, taking the real function $ \tilde{F} : \R^6 \to \R^6 $
 (cf. \ref{ss:comprel}), its  real
Jacobian  satisfies $ [ d \tilde{F} |_{\mathfrak{S}_\epsilon ^3}]=
\pi_3 ( \tau ) ( \Omega_{ \C } ( \Phi ) ) $ .

On the other hand we show that $ [ d \tilde{F}|_{\mathfrak{S}^3}] = \pi_3 (j) (\Omega(f)) $.
In order to recover the Smale invariant $ \Omega (f) $ of
$ f=\Phi|_{\mathfrak{S}^3}: \mathfrak{S}^3 \looparrowright S^5 $,
first we need to fix a global coordinate systems in  a neighbourhood of the source
$\mathfrak{S}^3$ in $\mathfrak{S}^5$
and also in $\R^5\approx S^5_\epsilon\setminus \{\mbox{a point}\}$ containing ${\rm im}(f)$. Let us introduce 
the `outward normal at the end' convention to orient  compatibly  a manifold and its boundary. 
In this way we fix an orientation of $\mathfrak{S}^5 = \partial\mathfrak{B}^6 $ and $ S^5 = \partial B^6 $. 
(According to \ref{ss:Smale}, the output of the proof is independent of the convention choice.)

In the first case we introduce a coordinate system
in $\mathfrak{S}^5\setminus \{Q\}\simeq \R^5$ compatibly with the orientation,
 where $ Q \in \mathfrak{S^5} \setminus \mathfrak{S}^3 $ is an arbitrary point (e.g.
 $(0,0,\rho)$).
 Let $ \nu' $ denote the framing of $ T (\mathfrak{S^5} \setminus \{Q \})\simeq T\R^5$ induced by
 this coordinate system. We can extend the outward normal frame $ \nu_6 $ of $ \mathfrak{S}^3 $ in $ \C^2 $ to the rest of $ \mathfrak{S}^5 \setminus \{Q \}  $ (as the outward normal vector
 of $\mathfrak{S}^5$). This framing can be extended to
a neighbourhood $ V $ of $ \mathfrak{S}^5 \setminus \{Q \}  $ in $ \C^3 $. Let $ \nu: V \to GL^+ (6, \R) $ denote this framing (or more precisely, $ \nu $ is the transition function from the standard framing inherited from $ \R^6 $ to the one just
constructed).

The target is the standard $ S^5 \subset \R^6 $. We can choose a point $ P \in S^5 \setminus f( \mathfrak{S}^3 ) $ and a coordinate system on $ S^5 \setminus \{P\} $ compatibly with the orientation. The coordinate system induces a framing $ \eta' $ of the tangent bundle $ T ( S^5 \setminus \{P\} ) $
of $ S^5 \setminus \{P\} $. In the points of the target of $ \tilde{F} $ the vectors of $ \eta' $ and $ d \tilde{F} ( \nu_6 ) $ are linearly independent, that is,
$d \tilde{F} ( \nu_6 ) $ behaves like a normal framing
(this follows from the transversality property of
\ref{ss:link}).
We can extend it to a normal framing $\eta_6$
of $ S^5 \setminus \{P\}  $ in $ \R^6 $. In this way we get a framing of the tangent bundle of a neighbourhood $ W $ of $ S^5 \setminus \{P\} $ in $ \R^6 $. Let $ \eta : W \to GL^+(6, \R ) $ denote the transition from the framing on $ W $ inherited from $ \R^6 $ to the framing just defined.

The Smale invariant $\Omega(f)$ is constructed in the following way, cf. \ref{ss:Smale}.
Take
\[ \mathcal{J}_{( \nu', \eta')} ( \tilde{F}|_{\mathfrak{T}}),  \]
the Jacobian of $ \tilde{F} $ restricted to $ \mathfrak{T}=F^{-1}(S^5_\epsilon) $ prescribed in the framings $ \nu' $ and $ \eta' $. The homotopy class of this matrix restricted to $ \mathfrak{S}^3 $ (as a map $ \mathfrak{S}^3 \to GL^+(6, \R) $) equals to $ \Omega(f) $. (Since $ \tilde{F} $ preserves the orientation, $ \tilde{F}|_{\mathfrak{T}} $ does as well.)
\[ \mathcal{J}_{ (\nu, \eta)} ( \tilde{F})|_{ \mathfrak{T} } = j ( \mathcal{J}_{ (\nu', \eta' )} ( \tilde{F}|_{ \mathfrak{T}})) \]
because $ d \tilde{F} (\nu_6) = \eta_6 $, thus the homotopy class of $ \mathcal{J}_{ (\nu, \eta ) } ( \tilde{F} )|_{ \mathfrak{S}^3} $ equals  $ \pi_3 (j) ( \Omega(f)) $.

On the other hand $ \mathcal{J}_{ ( \nu, \eta ) } ( \tilde{F}) = ( \eta^{-1} \circ \tilde{F}) \cdot d \tilde{F} \cdot \nu $. As maps $ \mathfrak{S}^3 \to GL^+(6, \R) $, $ ( \eta^{-1} \circ \tilde{F})|_{\mathfrak{S}^3} $ and $ \nu|_{\mathfrak{S}^3} $ are nullhomotopic because
 the vector fields are defined on the contractible spaces $\mathfrak{S}^5\setminus \{Q\}$ and
 $S^5\setminus \{P\}$. Therefore (cf. Remark~\ref{r:lie})
\[ [ \mathcal{J}_{ ( \nu, \eta ) } ( \tilde{F})|_{\mathfrak{S}^3} ]=
[( \eta^{-1} \circ \tilde{F})|_{\mathfrak{S}^3}] + [d \tilde{F}|_{\mathfrak{S}^3}] + [ \nu|_{\mathfrak{S}^3} ] = [d \tilde{F}|_{\mathfrak{S}^3}] \mbox{ .}
\]
The left hand side of this identity is $ \pi_3 (j) ( \Omega(f)) $, while the right
hand side  $ \pi_3 ( \tau ) ( \Omega_{ \C } ( \Phi ) ) $.
\end{proof}

\subsection{Proof of Theorem \ref{th:embintro}.} Part (a) follows from Theorem \ref{th:main} and
\cite{HM}.

In part (b),
the implications (1) $\Rightarrow$ (2,3,4), and (4) $\Rightarrow$ (3) are clear.

The proof of (2) $\Rightarrow$ (1): (2) implies $C(\Phi)=0$ by Theorem \ref{th:main}, while
this vanishing implies (1) via Mond's Theorem \ref{th:C}.
For (3) $\Rightarrow$ (1) we provide three proofs, each of them emphasize a different
geometrical/topological aspect.

{\it (A) (Based on Mumford's Theorem.)} \ If $f$ is an embedding then the image
$(X,0)$ of $\Phi$ is an
isolated hypersurface singularity in $(\C^3,0)$. Moreover, its link is $S^3$, hence by Mumford's
theorem \cite{mumford} $(X,0)$ is smooth. Hence its normalization $\Phi$ is an isomorphism.

{\it (B) (Based on Mond's Theorem.)} \   Let us take the generic deformation $\Phi_\lambda$,
and consider the preimage $D$ of the the set of double points. It is a 1--dimensional  closed
complex analytic subspace of the disc in $\C^2$. The preimages of cross cap and triple points
are interior points of the closure of $D$, while its boundary is $D\cap S^3$ is the preimage of the
double points of the immersion of $f:S^3\looparrowright S^5$. If $f$ is an embedding then
$\partial D=\emptyset$, hence $D$ is a compact analytic curve in (the disc of) $\C^2$,
hence it should be empty.  This shows that $\Phi_\lambda$ has no cross cap and triple points either.
Hence $C(\Phi)=0$, which implies (1) by \ref{th:C} as before.

{\it (C) (Based on Ekholm--Sz\H{u}cs Theorem.)} \ As above, we get that $\Phi_\lambda$ is an
embedding. Since $\Phi|S^3$ is an embedding, this embedding is regular homotopic to
$\Phi_\lambda|S^3$, hence they have the same Smale invariant. In the second case it
 can be determined by an
Ekholm--Sz\H{u}cs formula \cite{ESz} (recalled as Theorem \ref{th:ESz} here): since ${\rm im}(\Phi_\lambda)$
is an embedded Seifert surface with signature zero we get $\Omega(f)=0$.
This basically proves (3) $\Rightarrow$ (2). Then we continue with the already shown
(2) $\Rightarrow$ (1).

\vspace{2mm}

In fact, the main point of this last proof is already coded in Hughes--Melvin Theorem
\cite{HM} (\ref{th:HM} here),
but in that statement the Seifert surface is in $\R^5$ and not in $\R^6_+$ (or in the
6--ball). But \ref{th:ESz} shows that that Hughes--Melvin Theorem is true even if the
4--manifold $M^4$ with boundary in $\R^5$
is embedded in $\R^6_+$ (instead of $\R^5$).

\section{Examples}\labelpar{s:ex}

\subsection{} This section contains the first list of the promised examples.

\begin{ex}\label{ex:1} Fix $k \in \Z_{\geq 0} $. $ \Phi_{-k} (s, t) = (s, t^2, t^3 + s^k t ) $.
The ideal $ J $ (cf. \ref{ss:Wh}) is generated by  $ (2t, 3t^2 + s^k, -2kt^2 s^{k-1})= ( t, s^k ) $.
Hence $  \Omega (f) = - C( \Phi ) = -k $.

 This family gives representatives for every regular homotopy class with non-positive
 sign--refined Smale invariant. Furthermore, we can represent any regular homotopy class with
 Smale invariant $ k $ in the form $ \Phi_{-k} \circ \kappa $, where $ \kappa $ is the reflection
  $ \kappa (z, w) = (z, \bar{w}) $ (c.f. \cite[Lemma 3.4.2.]{ekholm3}).

\end{ex}

\begin{ex}[Singularities of type $ A $]\label{ex:A} These are quotient  singularities
 of the form $ (X,0) = (\C^2, 0)/ \Z_k $, where $ \Z_k=\{\xi\in \C\,|\, \xi^k=1\}$
  denotes the cyclic group of order $k$, and the action is $ \xi* (s, t) = ( \xi s, \xi^{-1} t )$
  for $ \xi \in \Z_k $. $(X,0)$ is the image of a map $\Phi$,
 whose components  are the generators of the invariant algebra $\C\{s,t\}^{\Z_k}$,
  see \cite[page 95]{invariant}, namely $ \Phi (s, t) = (s^k, t^k, st ) $. One can easily compute that
$ J = ( s^k, t^k, s^{k-1} t^{k-1}) $ and $ \Omega (f) = - C( \Phi ) = - (k^2 -1) $.
($(X,0)$ is the $A_{k-1}$--singularity.)
\end{ex}

\begin{ex}[Singularities of type $ D $]\label{ex:D} These are the quotient singularities of
 form $  (\C^2, 0)/ D_n $ where $ D_n $ denotes the binary dihedral group, \cite[page 89]{invariant}.
$ \Phi (s, t) = ( s^2 t^2, s^{2n} + t^{2n} , st (s^{2n} - t^{2n})) $ \cite[page 95]{invariant}.
By a computation
$ J = ( st (s^{2n}-t^{2n}) , s^2 t^2 (s^{2n} + t^{2n}) , (s^{2n} - t^{2n})^2 - 4n s^{2n} t^{2n}) $.
In singularity theory the quotient is known as the $D_{n+2}$--singularity.

A possible
 computation of $ \dim \left( \mathcal{O}_{\C^2, 0} / J \right) $ is based on the following facts.

\begin{lem}\label{lem:inter}
(a) Take $ f_1 , f_2, h \in \mathcal{O}_{\C^2, 0} $ such that $ f_1 f_2$ and $h $ are relative primes.
Then one has the  following exact sequence:
\[
 0 \to \mathcal{O}_{\C^2, 0} /( f_2 , h )  \to  \mathcal{O}_{\C^2, 0}
 /( f_1 f_2 , h )  \to  \mathcal{O}_{\C^2, 0} /( f_1 , h ) \to 0 \mbox{ .}
\]

(b) Take $ f_1 , f_2, g, h \in \mathcal{O}_{\C^2, 0} $ such that the ideal
$(f_1f_2, g,h) $ has finite codimension, and  $ h = f_1 h' $ for some $h'\in \mathcal{O}_{\C^2, 0} $.
Then one has the following exact sequence:
\[
 0 \to\mathcal{O}_{\C^2, 0} /( f_2 , g, h' )  \to  \mathcal{O}_{\C^2, 0} /
 ( f_1 f_2 ,g,  h )  \to  \mathcal{O}_{\C^2, 0} /( f_1 , g )  \to 0 \mbox{ .}\]
\end{lem}

\begin{proof}
Part (a) is well-known as the additivity property of the local intersection number of plane curves,
see e.g. \cite{fulton}. The proof of part (b) is similar.
\end{proof}

 Using these lemmas the codimension of $ J $ can be calculated, and it is $ 4n^2 + 12n -1 $.
 Hence, the Smale invariant of the covering $S^3\to \{\mbox{link of the $ D_{n+2}$--singularity}\}$
  is $ - ( 4n^2 + 12n -1 ) $.
\end{ex}

\begin{example}\label{ex:wh}
Assume that the three components of $\Phi$ are weighted homogeneous of weights
$w_1$ and $w_2$ and degree $d_1$, $d_2$ and $d_3$. Then, cf. \cite{Mondwh},
$$C(\Phi)=\{d_1d_2+d_2d_3+d_3d_1-(w_1+w_2)(d_1+d_2+d_3-w_1-w_2)-w_1w_2\}/w_1w_2.$$
Mond proved this identity for germs with finite right--left codimension, but the same proof works for
germs with finite ${\mathcal O}_{\C^2,0}/J$.

For example, if $\Phi: (\C^2,0)\to (\C^2,0)/G\hookrightarrow (\C^3,0)$ is as in Example
\ref{ex:ade}, then all three components are homogeneous ($w_1=w_2=1$). In the case of
$A_{k-1}$ and $D_{n+2}$ the degrees are $(k,k,2)$ and $(4,2n,2n+2)$ respectively.
Hence the values $C(\Phi)$ from Examples \ref{ex:A} and \ref{ex:D} follow in this way as well.

For $E_6,\ E_7$ and $E_8$ singularities the degrees are
$(6,8,12)$, $(8,12,18)$ and $(12, 20,30)$ respectively, see \cite[4.5.3--4.5.5]{invariant},
hence the corresponding values
$-\Omega(f)$ are 167, 383, 1079.
\end{example}

\section{Smale invariant via Seifert surfaces}\labelpar{s:ss}

\subsection{} In this section we review three important topological formulae targeting the
Smale invariant in terms of the geometry of oriented Seifert surfaces.
They are stated and proved only
 up to a sign ambiguity. In the next section we will  show that
the sign--refined Smale invariant 
appears in all these expressions with a unique
well--defined sign, and we determine
it simultaneously for all  formulae. The discussion  has an extra output as well:
the topological ingredients in the formulae below
will get reinterpretations in terms of complex analytic invariants, provided that the immersion is
induced by a holomorphic germ $\Phi$.

In the spirit of the discussion of subsection \ref{ss:Smale}, in this section we will
write ${\bf S}^3$ for an `oriented abstract $S^3$'. $\Omega(f)$ will denote 
the Smale invariant (given by any of its definitions, still having its sign--ambiguity). 
Note that in the next statements we need to fix a `boundary convention', in order to have the notion 
of oriented $\partial M$. (Nevertheless, the sign--corrected formulae will be `boundary convention' free,
cf. Theorem \ref{thm:new}.)

\begin{thm}[Hughes, Melvin \cite{HM}]\label{th:HM}
 Let $ f: {\bf S}^3 \hookrightarrow \R^5 $ be an embedding and $ \tilde{f}: M^4 \hookrightarrow \R^5 $ be a Seifert surface of $f$, i.e. $ M^4 $ is a compact oriented $4$-manifold with boundary
  $ \partial M^4 = {\bf S}^3 $ and $ \tilde{f} $ is an embedding such that $ \tilde{f}|_{ \partial M^4} = f $.
 Let $\sigma(M^4)$ be the signature of  $M^4 $. Then
\begin{equation}\label{eq:hm}
\Omega (f) = \pm \frac{3}{2} \sigma (M^4) \mbox{ .} \end{equation}
\end{thm}

For arbitrary immersions Ekholm and Sz\H{u}cs generalized the formula
via generic \emph{singular Seifert surfaces}, and
in two different ways:   mapped either in $\R^5$ or in $\R^6_+$ \cite{ESz},
 see also
 \cite{Esz2,saeki}.


 If $M^4$ is a compact oriented $4$-manifold and $ g: M^4 \to \R^5 $ is a generic $C^\infty$ map, then
  $g$ has isolated \emph{$ \Sigma^{1, 1} $--points (cusps)}, each endowed with a well--defined sign.
  Let $\#\Sigma^{1,1}(g)$ be their `algebraic' number (cf.  \cite{ESz}).

\begin{thm}[Ekholm, Sz\H{u}cs \cite{ESz}] Let $ f: {\bf S}^3 \looparrowright \R^5 $ be an immersion 
and $ M^4 $ be a compact oriented $4$-manifold with boundary ${\bf  S}^3 $. Let 
$ \tilde{f}: M^4 \to \R^5 $ be a generic map such that
$ \tilde{f}|_{ \partial M^4} $ is regular homotopic to $ f $ and $ \tilde{f}$ has no singular points 
near the boundary. Then 
\begin{equation}\label{eq:cuspos}
\Omega (f) = \pm \frac{1}{2} (3 \sigma (M^4) + \# \Sigma^{1, 1} (\tilde{f}))  \mbox{ .} \end{equation}
\end{thm}
The last formula, the most important from the point of view of this note,
uses generic $C^\infty$ maps $ g: M^4 \to \R^6 $ defined on  compact oriented $4$-manifolds $M^4$.
It involves three topological invariants associated with such a map.
Next we review their definitions. They will be computed for two concrete holomorphic  maps
in order to identify the missing sign.

 If $g$ is as above, then it
 has isolated triple values (three local sheets of $ M^4$ intersecting in general position).
 Such a point is endowed with a well--defined sign \cite[2.3]{ESz}).
\begin{defn}[\cite{ESz}]
 $t(g) $ denotes the algebraic number of the triple values of $g$.
\end{defn}

Next, assume that $ \partial{M^4} = {\bf S}^3$ and 
$g: (M^4, \partial M^4) \to (\R^6_+, \partial \R^6_+ )$
is generic, it is  nonsingular near the boundary, and  $ \tilde{f}^{-1}(
\partial \R^6_+ )= \partial M^4 $.
 Here  $\R^6_+ $ is the closed half--space of $\R^6$.
 The set of  double values of $ g$ is an immersed oriented $2$-manifold,  denoted by $D(g)$.
 Its oriented boundary consists of two parts, the intersection
  of $D(g)\cap \partial \R^6_+$, and the other,
  disjoint with $\partial \R^6_+$, is the set of singular values $\Sigma(g)$ of $ g $.
 Let $ \Sigma'(g) $ be a copy of $ \Sigma(g) $ shifted slightly along the outward normal vector field
 of $ \Sigma(g) $ in $D(g)$. Then $ \Sigma'(g) \cap g(M^4) = \emptyset $.

\begin{defn}[\cite{ESz}]\label{d:l}
 $ l(g) $ denotes the linking number of $g(M^4) $ and $ \Sigma'(g) $ in $ (\R^6_+, \partial \R^6_+ ) $.
\end{defn}

For a generic (self-transverse) immersion $ f: {\bf S}^3 \looparrowright \R^5 $ one defines
 an integer $L(f)$ as follows
 \cite[2.2]{ESz}, \cite[2.2]{saeki}.
 $ f$ has a normal framing $ (v_1, v_2) $ which is unique up to
 homotopy. In any double value $ y= f(x_1) = f(x_2) $ set $ N(y)= v_1(x_1) + v_1(x_2) $.
 Let $ D'(f) $ be a copy of the set of double values $ D(f) $ of $f $ shifted slightly along the vector field $ N$. $ D(f)$ (hence  $D'(f) $ too)  is a $1$-manifold and 
 $ D'(f) \cap f({\bf S}^3) = \emptyset $.

\begin{defn}[\cite{ekholm3,ESz,saeki}]\label{d:L}
 $L(f)$ is a the linking number of $ f({\bf S}^3)$ and $ D'(f) $ in $ \R^5 $.
\end{defn}

\begin{thm}[Ekholm, Sz\H{u}cs \cite{ESz}]\label{th:ESz}
 Let $ f: {\bf S}^3 \looparrowright \R^5 $ be an immersion and $ M^4 $ be a compact oriented $4$-manifold with boundary $ \partial M^4 = {\bf S}^3 $. Let $ \tilde{f}: (M^4, \partial M^4) \to (\R^6_+, \partial \R^6_+ )$ be a generic map nonsingular near the boundary, such that $ \tilde{f}^{-1}( \R^6_+ )= \partial M^4 $ and $ \tilde{f}|_{\partial M^4} $ is regular homotopic to $ f $. Then
 \begin{equation}\label{eq:hurk}
 \Omega (f) = \pm \frac{1}{2} (3 \sigma (M^4) + 3 t(\tilde{f}) - 3 l(\tilde{f}) + L(f))  \mbox{ .} \end{equation}
\end{thm}

\section{Ekholm-Sz\H{u}cs formulae for holomorphic germs $\Phi$}\labelpar{s:eszcomp}

In this section, from a holomorphic deformation of $\Phi$ we
construct a singular Seifert surface,
and we express the topological
 summands of (\ref{eq:hurk}) in terms of
holomorphic invariants. As a corollary we specify the
sign in the formulae (\ref{eq:hm}), (\ref{eq:cuspos}) and (\ref{eq:hurk}).

\subsection{Singular Seifert surface associated with an analytic deformation}\labelpar{ss:assoc} \

Let $ \Phi: (\C^2, 0) \to (\C^3, 0) $ be a holomorphic germ singular only at the origin and let $ f: S^3 \looparrowright S^5 $ be the immersion associated with $ \Phi $. We take an $ \epsilon $ as in Corollary~\ref{cor:epsilon}, that is, we fix in the target a ball $B^6_\epsilon$.
We also consider a holomorphic  generic deformation $ \Phi_{\lambda} $ of $ \Phi_0 = \Phi $,
and we fix $\lambda$ sufficiently small,  $ 0 < |\lambda| \ll \epsilon $, such that
 the cross caps and (if $ T( \Phi) < \infty$) the triple points of $ \Phi_{\lambda} $ sit in $ B^6_{\epsilon} $.
We set
$ \mathfrak{B}^4_{\epsilon, \lambda} := \Phi_{\lambda}^{-1} (B^6_{\epsilon}) $, it is a $C^\infty$
 non--metric ball in $\C^2$. Its  boundary is  $ \mathfrak{S}^3_{\epsilon, \lambda} :=
 \Phi_{\lambda}^{-1} (S^5_{\epsilon}) $,  it is canonically diffeomorphic to $ S^3 $.

\vspace{2mm}


%
The map $\Phi_\lambda$ is generic as a holomorphic map, but it is not
generic as a $C^\infty$ map. The $C^\infty$ genericity is obstructed by its cross cap points.
We will modify $\Phi_\lambda$ in the neighborhood of these points according to the
following local model.

Let us fix local holomorphic coordinate systems in the source and the target  such that
$\Phi_\lambda$ in the neighborhood of a cross cap has local equation $ \Phi^{loc} (s, t) = (s^2, st, t)$.
We consider its real smooth deformation (with $0\leq \tau \ll |\lambda|$):
\begin{equation}\label{eq:whpert}
\Phi^{loc}_\tau(s, t)=(s^2 + 2 \tau \bar{s}, st + \tau \bar{s}, t).
\end{equation}
Since the restriction of $\Phi^{loc}$ near the boundary of the local 4-ball is stable,
by a $C^\infty$ bump function the local deformation can be glued to the trivial deformation of
$\Phi_\lambda$ outside of local neighborhoods of the cross caps. This gives a $C^\infty$ global
deformation $\Phi_{\lambda,\tau}$ of $\Phi_\lambda$ and $\Phi$.
The map
$ \tilde{f}=\Phi_{\lambda,\tau}:  (\mathfrak{B}^4_{\epsilon, \lambda}, \mathfrak{S}^3_{\epsilon, \lambda} )
  \to (B^6_\epsilon, S^5_\epsilon) $ is the  singular Seifert surface we will consider.
Its restriction,  $ f_{\lambda} = \Phi_{\lambda,\tau} |_{ \mathfrak{S}^3_{\epsilon, \lambda}}
 =\Phi_{\lambda}|_{ \mathfrak{S}^3_{\epsilon, \lambda}}$ is
  the immersion associated with $ \Phi_{\lambda} $.

\begin{prop}\label{l:gen} \

(a)  $ \tilde{f} : {\mathfrak B}^4_{\epsilon,\lambda} \to \C^3 $
is a generic smooth map, nonsingular near the boundary.

 (b) $ f_{\lambda} $ is a generic immersion and it is regular homotopic to $ f$.

 (c) If $ f $ is a generic immersion, then $ f_{\lambda} $ is regular homotopic to $f$ through generic immersions. In this case $ L(f) = L ( f_{\lambda}) $.
\end{prop}

\begin{proof}
(a) First one checks that the local $\Phi^{loc}_\tau$
 is generic. This follows from the  computation from
 section~\ref{ss:lwh}. 
  Its most complicated singularities are $ \Sigma^{1, 0} $ (fold) points,
  the singular values constitute  an $ S^1 $, which -- together with the double values of the image
  of the boundary of the local ball -- bounds the $2$-manifold of the double values. Cf.
  \cite[2.3.]{ESz}.

In the complement of local balls $\Phi_{\lambda,\tau} $ agrees with $ \Phi_{\lambda} $, hence it has only
simple points, self-transverse double points and isolated triple points. All of them are generic.
Hence $\tilde{f}$ has all the local property of a generic map (and, in fact, this is enough
in the determination of all the invariants, cf. \cite{ESz}).

(b) $\Phi_\lambda|_{ \mathfrak{S}^3_{\epsilon, \lambda}} $ is generic in real sense too:
 it has only simple points and generic self-transversal double points.
 $\Phi_{h\lambda}|_{ \mathfrak{S}^3_{\epsilon, \lambda}} $
 is a  regular homotopy between $ f $ and $ f_{\lambda} $ ($ h \in [0, 1] $).

(c) Being a generic immersion is an open condition (cf. \cite[2.1.]{ESz}).
Furthermore, $ L $ is constant along a regular homotopy through generic immersions, cf. \cite{ekholm3}.
\end{proof}

Next,  we return back to the formula (\ref{eq:hurk}), applied for $\tilde{f}$.
  Clearly, $\sigma(M^4)=0$.

\begin{thm} \

(a) $ t( \tilde{f}) = T ( \Phi_{\lambda}) $  (cf. \ref{ss:trip}).

(b)  $ l(  \tilde{f} ) = C (\Phi) $.

(c) $ L(f_{\lambda}) = C( \Phi) - 3 T ( \Phi_{\lambda}) $.

\noindent
In particular, if\, $T(\Phi)<\infty$, hence\, $ T ( \Phi_{\lambda}) $ is independent of the deformation
$ \Phi_{\lambda}$, then $ t( \tilde{f}) = T ( \Phi) $
 and $ L(f_\lambda ) = C( \Phi) - 3 T ( \Phi)$ is also independent of the deformation.
\end{thm}
\begin{thm}\label{thm:new} With our sign--convention, if in the left hand side of the formulae
(\ref{eq:hm}), (\ref{eq:cuspos}) and (\ref{eq:hurk}) we put the sign--refined Smale invariant
$\Omega^a(f)$, then
the formulae are valid if we put the positive sign on the right hand sides.

In particular, the validity of these sign--corrected formulae (e.g., $\Omega^a(f)=\frac{3}{2}\sigma(M^4)$)
is independent of the `boundary convention': changing the boundary convention changes the sign in both 
sides of the formulae simultaneously. 
\end{thm}
We prove both theorems simultanously (see also the discussion from subsection \ref{ss:Smale}). 
\begin{proof}

In the definitions of the invariants $t$, $l$ and $L$ one uses very specific sign/orientation
conventions, based on the orientation of the involved subspaces in their definition.

For a triple value, the sign is determined in such a way that it is $+1$ whenever the triple
value is obtained from a holomorphic triple point (hence the orientations agree with
the complex orientations).

Since in the local deformation $\Phi^{loc}_\tau$ we do not create any new triple value,
see e.g. the computation of section \ref{s:calc}, all the triple values of $\tilde{f}$
come from the complex triple points of the holomorphic $\Phi_\lambda$, hence (a) follows.

The proof of the remaining parts are based on computations of the invariants
$C(\Phi)$, $T(\Phi)$, $l(\tilde{f})$ and $L(f)$ for two concrete cases.
For the integers $l$ and $L$ the definitions (orientation conventions)
are not immediate even in simple cases. Therefore, in our computation
we determine them only up to a sign.
The point is that computing `sufficiently many' examples, the formula (\ref{eq:hurk}),
even with its sign ambiguity in front of the right hand side, and even with
the (new) sign ambiguities of the integers $l$ and $L$, determine uniquely all these signs.
(This also shows  that, in fact, there is a unique
universal way to fix the orientation conventiones and signs in the definitions of
$l$ and $L$ such that  (\ref{eq:hurk})
works universally.)

In section \ref{s:calc} we will determine the following data:
\begin{equation}\label{eq:CrC}\begin{split}
 \mbox{(i)}\ \mbox{For cross cup:} \ \ \ C(\Phi)=1, \ T(\Phi)=0, \ l=\pm 1, \ L=\pm 1.\\
\mbox{(ii)} \ \ \ \ \ \ \ \ \ \
\mbox{For $A_1$:} \ \ \ C(\Phi)=3, \ T(\Phi)=1, \ L=0.\ \ \ \ \ \hspace{1cm}
\end{split}
\end{equation}

(b)
The singular values of $ \tilde{f} $ are concentrated near the cross caps of $ \Phi_{\lambda} $.
For $\Phi^{loc}_\tau$ the value  $ l $ is $ \pm 1 $, see (i).
Since the sign is the same for all cross caps, $ l( \tilde{f})= \pm C( \Phi ) $.

We introduce the notation
\begin{equation}\label{eq:omvessz}
 \Omega' (f_{\lambda}) :=  \frac{1}{2} ( 3 t(\tilde{f}) - 3 l(\tilde{f}) + L(f_{\lambda})).
\end{equation}

$ \Omega'(f_{\lambda}) $ agrees with $ \Omega(f) $ up to sign, thus $ \Omega'(f_{\lambda}) =
\pm C( \Phi) $. Substituting this and  the data (i) of the cross cap in (\ref{eq:omvessz})
 we conclude that $ l(\tilde{f}) = - \Omega' (f_{\lambda}) $ and
  $ L(f_{\lambda}) = \pm C( \Phi ) -  3 T ( \Phi_{\lambda})$.

Next, using the date (ii) for  $ A_1 $, all the remaining sign ambuguities can be eliminated:
$ L(f_\lambda) = C( \Phi ) -  3 T ( \Phi_\lambda)$, $ l( \tilde{f})= C( \Phi ) $ and
$ \Omega'(f_{\lambda}) = - C( \Phi) = \Omega (f) $.

The universal signs in formulae
(\ref{eq:hm}), (\ref{eq:cuspos}) and (\ref{eq:hurk}) are related by common examples, hence
one of them determines all of them.
\end{proof}

\begin{remark}
The formula (\ref{eq:cuspos}), involving the (algebraic) number of real cusps of maps $g:M^4\to \R^5$
is the real analogue of our theorem $\Omega(f)=-C(\Phi)$, involving the
number of  cross caps of holomorphic deformations. This suggests
that if we replace a holomorphic deformation by a smooth generic map, then we
trade each cross cup by $-2$ real cusps.
\end{remark}

\section{Calculations. The proof of (\ref{eq:CrC}).}\labelpar{s:calc}

We show the main steps of the computations, with their help the reader can
fill in the details.
Note that if
the germ $ \Phi $ is weighted homogeneous, then $ \epsilon_0 =1 $ can be chosen.

\subsection{The case of cross cap.}\labelpar{ss:lwh}

For the computation of $T(\Phi)$ see e.g. \cite{Mond1,Mond2}; $C(\Phi)$ is clear. Next we compute $l$ and $L$.
Set $ \Phi(s, t)= (s^2, st, t) $ and the  smooth perturbation
$ \tilde{f}(s, t)=(s^2 + 2 \epsilon \bar{s}, st + \epsilon \bar{s}, t)$.
The singular locus is
$\tilde{ \Sigma }=  \{ (s,t) \ | \ s=t \mbox{ , } |s|=|t|= \epsilon \} \cong S^1$.

$ \tilde{f}|_{\Sigma}$ has no singular point, hence  $ \tilde{f}$  has no cusp points.
The most complicated singularities of $ \tilde{f} $ are $ \Sigma^{1, 0} $ (or fold) points.
The set of the double points of $ \tilde{f} $ is
\[ \tilde{D} =
\{ (s, t) \in \mathbb{C}^2 \ | \ (s - t)t + \epsilon ( \bar{s} - \bar{t}) = 0 \}
\setminus \{s=t\}.
\]
with the involution  $(s, t) \mapsto (s', t)= (2t-s, t)$.
The fix point set of the involution is $\{s=t\}$.
Each double point has exactly one pair with the same value, hence
$ \tilde{f} $ has no triple point.

 A parametrization of $ \tilde{D}$ is
$
 (\rho, \alpha) \mapsto ( -\epsilon e^{-2 \alpha  i } + \rho e^{i \alpha} , -\epsilon e^{-2 \alpha  i }),
$
where $ \rho \in \R_+ $, $ \alpha \in [0, 2 \pi) $.

The parametrization shows that the closure of $ \tilde{D} $ is a M\"{o}bius band.
For $ \rho=0 $ we get $ \tilde{\Sigma} $, which is the midline of the M\"{o}bius band.
 The set of double values is
\begin{align*} D= \tilde{f} (\tilde{D}) &= \{ (s^2 + 2 \epsilon \bar{s}, st + \epsilon \bar{s}, t) \ | \ (s, t) \in \tilde{D} \} \\
                            &= \{ ( \rho^2 e^{2 i \alpha}+ \epsilon^2 e^{2 i \alpha } ( e^{-6 i \alpha }-2 ),
\epsilon^2 ( e^{ - 4 i \alpha } - e^{ 2 i \alpha } ) ,
-\epsilon e^{-2 \alpha  i } ) \ | \ \rho \in \R_+ \mbox{ , } \alpha \in [0, 2 \pi) \}.
\end{align*}
Writing $ \rho = 0 $ we get the singular values of $ \tilde{f} $,
\[ \Sigma= \tilde{f} (\tilde{\Sigma} ) =
\{ ( \epsilon^2 e^{2 i \alpha } ( e^{-6 i \alpha }-2 ),
\epsilon^2 ( e^{ - 4 i \alpha } - e^{ 2 i \alpha } ) ,
-\epsilon e^{-2 \alpha  i } ) \} \mbox{ .}
\]

The inward normal field of $ \Sigma  $ in $ D $ is the derivative of the curve
\[ \gamma(t)=  (t e^{2 i \alpha} + \epsilon^2 e^{2 i \alpha } ( e^{-6 i \alpha }-2 ),
\epsilon^2 ( e^{ - 4 i \alpha } - e^{ 2 i \alpha } ) ,
-\epsilon e^{-2 \alpha  i } )  \]
at $ t=0 $, that is $\gamma'(t)|_{t=0} = (e^{2 i \alpha}, 0, 0)$.
The pushing out of $ \Sigma$ (cf. Definition~\ref{d:l}) is
\[ \Sigma'= \Sigma - \delta \cdot \gamma'(t)|_{t=0}
= \{ ( - \delta e^{2 i \alpha } + \epsilon^2 e^{2 i \alpha } ( e^{-6 i \alpha }-2 ),
\epsilon^2 ( e^{ - 4 i \alpha } - e^{ 2 i \alpha } ) ,
-\epsilon e^{-2 \alpha  i } ) \} \mbox{ ,}
\]
where $ 0 < \delta \ll \epsilon $.
By Definition~\ref{d:l} we need the linking number of $ \tilde{f} ( \R^4) $ and $ \Sigma' $ in $ \R^6 $. To calculate it we fill in $ \Sigma' \cong S^1 $ with a `membrane', which here will be   the disc
\[ H= \{ (- \delta w + \epsilon^2 (\bar{w}^2-2 w), \epsilon^2 (\bar{w}^2-w), - \epsilon \bar{w} ) \ | \ w \in \C \mbox{ , } |w| \leq 1 \}.
\]
$l(\tilde{f})$ is the algebraic number of the intersection points of $ H $ and $ \tilde{f} ( \R^4) $.
The  only solution is $ w=0$, $ (s, t)=(0,0) $, and the
intersection at this point is transversal.
Hence, for the smooth perturbation $ \tilde{f} $ of the cross cap $ l( \tilde{f}) = \pm 1 $.

Next we compute $L$.
The set of the double points of $ \Phi $ is
$\tilde{D}= \{ (s, 0) \ | \ s \neq 0 \} \subset \mathbb{C}^2$.

The set of the double values is
$
D= \Phi(\tilde{D}) = \{ (s^2, 0, 0) \ | \ s \neq 0 \} \subset \mathbb{C}^3$,
and the set of the double values of $ f $ is
$
 D_f = D \cap S^5 = \{ (s^2, 0, 0) \ | \ | s | =1 \} \subset S^5$.

The sum of the normal vectors at $ (s^2, 0, 0 ) $ is $(0,0,\bar{s}^2)$. Hence
the shifted copy of $ D $ along $ N $ is
$D'= D_f + \delta N = \{ (s^2, 0, \delta \bar {s}^2 ) \ | \ | s | =1 \}$.

Since $ D' $ does not intersect $ \Phi (\C^2 ) $ for $ \delta \in (0, 1] $,
 we can choose $ \delta = 1 $.
An injective parametrization of $ D_f + \delta N $ is
$
 D' = \{ (z, 0,  \bar {z} ) \ | \ | z | =1 \},
$
where $ z=s^2 $.
To calculate the linking number of $ \Phi( \C^2 ) $ and $ D' $ in $ \R^6 $, we need a
membrane which fills in $ D $. We take
\[
 H = \{ (z, \sqrt{1-|z|^2},  \bar {z} ) \ | \ | z | \leq 1 \} \cong D^2 \mbox{ .}
\]
$L(f)$ is the algebraic number of the intersection points of $ \Phi( \C^2 ) $ and $ H $.
But there is only one such point, namely
$P:= \Phi ( \sqrt{\xi}, \xi) = ( \xi, \xi \sqrt{\xi}, \xi )$,
where $\xi$ is the real root of  $ g(z) :=z^3 + z^2 -1 =0 $.
Moreover, this intersection is transversal.

\subsection{The $A_1$ singularity}\labelpar{ss:A} By \ref{ex:A} it is
given by $ \Phi_{0} (s, t) = (s^2, t^2, st) $. The immersion $ f_0$ associated with $ \Phi_0 $ is not generic, $f_0$ is the $2$-fold covering of the projective space composed with the inclusion. Thus all points of $ S^3 $ are double points of the immersion $ f$.

On the other hand, by \ref{ex:A}, $ C( \Phi_0 ) = 3 $, and a similar calculation of the codimension of the second fitting ideal shows that $ T( \Phi_0) = 1 $. The finiteness of these invariants shows that the number of cross caps and  triple points of a generic deformation of $ \Phi_0 $ are independent
of the chosen deformation. Below we give a concrete deformation $ \Phi_{\epsilon}$ of $ \Phi_0 $ and we calculate the invariant $ L $ of the generic immersion $ f_{ \epsilon} $ associated with $\Phi_{\epsilon}$.

The deformation is $\Phi_{\epsilon}(s, t) = ((s- \epsilon ) s , (t- \epsilon ) t , st ) $.
The vector field
\[ \tilde{N} (s, t) = \overline{  \partial_s \Phi_\epsilon (s, t) \times \partial_t \Phi
_\epsilon (s, t) } =
\left( \begin{array}{c}
\bar{t} (2 \bar{t} - \epsilon ) \\
- \bar{s} (2 \bar{s} - \epsilon ) \\
(2 \bar{s} - \epsilon ) (2 \bar{t} - \epsilon )
\end{array} \right) \mbox{ .} \]
is $ 0 $ at the points $(0, \epsilon /2 )$, $( \epsilon/2, 0 ) $ and $( \epsilon/2, \epsilon/2 ) $. These  are the cross caps.

The defining equation $ \Phi_\epsilon (s, t) = \Phi_\epsilon (s', t' ) $ (where $(s, t) \neq (s', t') $) of the double points leads to the system of equations
\[
 (s-s')(s+s'- \epsilon ) = 0, \ \
 (t-t')(y+y' - \epsilon ) = 0, \ \
 st=s't'.
\]
Thus the double locus $ \tilde{D} $  has three parts and these parts correspond to the three cross caps.
The first part comes from the solution $ s'=s $ and $ t'= \epsilon - t $, which implies $ s=0 $,
hence $
 \tilde{D}_1 = \{ (0, t) \ | \ t \neq \epsilon/2 \}$
with $ \Phi_\epsilon (0, t) = \Phi_\epsilon (0, \epsilon - t) $. This provide the double value set
\[
 D_1 = \Phi_\epsilon ( \tilde{D_1} ) = \{ (0, t (t- \epsilon ) , 0) \ | \ t \neq  \epsilon/2 \}.
\]
The second part comes from the solution $ s'= \epsilon - s $ and $ t'=  t $, which implies $ t=0 $,
and $
 \tilde{D}_2 = \{ (s, 0) \ | \ s \neq \epsilon/2 \}$
with $ \Phi_\epsilon (s, 0) = \Phi_\epsilon (\epsilon - s, 0) $. The set of double values is
\[
 D_2 = \Phi_\epsilon ( \tilde{D_2} ) = \{ (s (s- \epsilon ) , 0,  0) \ | \ s \neq  \epsilon/2 \}.
\]
The third part comes from the solution $ s'= \epsilon - s $ and $ t'= \epsilon - t $, which implies $ s + t = \epsilon $, and
$
 \tilde{D}_3 = \{ (s, \epsilon - s)\} \ | \ s \neq \epsilon/2 \}$
with $ \Phi_\epsilon (s, \epsilon - s) = \Phi_\epsilon (\epsilon - s, s) $. The set of double values is
\[
 D_3 = \Phi_\epsilon ( \tilde{D_3} ) = \{ (s (s- \epsilon ) , s (s- \epsilon ),  -s (s- \epsilon )) \ | \ s \neq  \epsilon/2 \}.
\]
$ D_1 $, $ D_2 $ and $ D_3 $ intersect each other in the unique triple value
$ \Phi_\epsilon (0, 0) = \Phi_\epsilon ( \epsilon, 0) = \Phi_\epsilon (0, \epsilon ) = (0, 0, 0) $.

Let $ D_i(f) = D_i \cap S^5 $ ($i=1, 2, 3$) denote the disjoint components of the set of the double
values of $ f$. Clearly $ L(f) = L_1(f) + L_2 (f) + L_3 (f) $, where $ L_i (f) $ is the linking number corresponding to the component $ D_i(f) $.
But $ L_1(f)=L_2(f)=L_3(f)$. Indeed,
$ D_1 $ and $ D_2 $ is interchanged via the
 transformations $ \phi (s, t) = (t, s) $ (of $\C^2$)
  and $ \psi (X, Y, Z) = (Y, X, Z) $ (of $\C^3$), and
   $ D_3 $ and $ D_2 $ via $ \phi(s, t) = ( \epsilon - s - t, t) $ and $ \psi (X, Y, Z) = (X+Y+2Z, Y, -Y-Z)$. Thus, it is enough to calculate $ L_1(f) $.
The needed vector field along $ D_1 $ is
\[ N (0, t (t- \epsilon ), 0) = \tilde{N} (0, t) + \tilde{N} (0, \epsilon - t) =
((2 \bar{t} - \epsilon )^2,0,0).\]
The set of the double values of $f$  corresponding to $D_1$ is
\[
 D_1 (f) = D_1 \cap S^5 = \{ (0, t (t- \epsilon ) , 0) \ | \ |t (t- \epsilon )| = 1 \} \mbox{ .}
\]
The shifted  $ D_1 (f) $ along $ N$ is
\[
D'_1 = D_1 (f) + \delta N = \{ (\delta (2 \bar{t} - \epsilon )^2, t (t- \epsilon ) , 0) \ | \ |t (t- \epsilon )| = 1 \} \mbox{ ,}
\]
where $ \delta $ is small enough. Nevertheless, we can choose $ \delta = 1 $, because
$ D'_1  \cap \Phi ( \C^2 ) = \emptyset $ for any $ \delta \in (0, 1] $.  With the notation $ z = t(t- \epsilon ) $ we give an injective parametrization
$
D'_1 = \{ (4 \bar{z} + \epsilon^2 , z , 0) \ | \ |z  | = 1\}$.
We fill it with the membrane
\[
H = \{ (4 \bar{z} + \epsilon^2 , z , i \sqrt{1- |z|^2} ) \ | \ |z  | \leq 1 \} \mbox{ .}
\]
Computing the intersection points of $ H $ and $ \Phi( \C^2 ) $ leads to the equations
\[
4 \bar{z} + \epsilon^2 = a (a - \epsilon ), \ \
z = b (b - \epsilon ), \ \
i \sqrt{1- |z|^2} = ab,
\]
with $ | z | \leq 1 $ and $ \epsilon $ small.
The first two equations imply that $ |a| < 5 $ and $ |b| < 2 $. Multiplying the first two equations one gets
\[
z (4 \bar{z} + \epsilon^2 ) = a^2 b^2 - a^2 b \epsilon - a b^2 \epsilon + ab \epsilon^2.
\]
From the third equation follows $ a^2 b^2 = |z|^2 - 1 $, hence
\[
3 |z|^2 = -1 -  z \epsilon^2 - a^2 b \epsilon - a b^2 \epsilon + ab \epsilon^2,
\]
and the right hand side is negative if $ \epsilon $ is small enough. Hence $ H \cap \Phi ( \C^2 ) = \emptyset $, and $L(f)=0$.

\section{Final remark. The real version in arbitrary dimension.}


\subsection{}
 There is a real version of  part (b) of Theorem~\ref{th:main} which follows directly from the result of Whitney and Smale.

 Let $ \Phi: (\R^{n+1},0)  \to (\R^{2n+1}, 0) $ be a real  analytic germ singular only at $0$.
  With the same method as in the complex case we can associate  an immersion $ f: S^n \looparrowright S^{2n} $ with $ \Phi$ (see \ref{ss:link}). A generalization of Whitney's double point formula valid for plane curve immersions \cite{Whitney} shows that the Smale invariant of $ f $
  (more precisely, of  a generic immersion regular homotopic to $f$) equals the algebraic number of self-intersection points (${\rm mod}\ 2 $ if $ n $ is odd).

   A generic perturbation $ \Phi' $ of $ \Phi $ has only cross cap type singularities, i.e.
   locally right--left equivalent with germs of
   the form $ (s, t) \mapsto (s^2, st, t) $, where $ s \in \R $ and $ t \in \R^n $.
   These cross caps are isolated, and if $n$ is even, we can associate a sign for each of them. $ \Phi' $ restricted to the boundary is a generic immersion $ f': S^n \looparrowright S^{2n} $. $f'$ and $f$ are regular homotopic, and $f'$ has two kinds of double values:

   (a) double values related to a cross cap (that is, they are connected by a segment consisting
   of double values of $\Phi'$),

   (b) double values not related to a cross cap.

   When $n$ is even, the sign associated to a cross cap agrees with the sign associated with the self intersection point of $ f'$ related to the cross cap. Thus the algebraic number of such
   cross caps is equal to the algebraic number of double values of type (a) ($ {\rm mod}\ 2$ if $ n $ is odd). The double points of type (b) are pairwise joined up by segments of the double values of $ \Phi' $, thus the algebraic number of them is $ 0$.
   Moreover, it can happen that two cross caps are joined by a segment consisting of double values of $\Phi'$,
   but then they will have different algebraic sign, hence they will not contribute in the sum. 
   Hence, we proved:

   \begin{prop}
   The Smale invariant of $f$ agrees with the algebraic number of the cross cap points appearing in a generic perturbation of $ \Phi $  (${\rm mod}\ 2 $ if $ n $ is odd).
   \end{prop}


\begin{thebibliography}{10}

\bibitem{bottper}
Raoul Bott.
\newblock The stable homotopy of the classical groups.
\newblock {\em Annals of Mathematics}, 70(2):313--337, 1959.

\bibitem{ekholm3}
Tobias Ekholm.
\newblock Differential 3-knots in 5-space with and without self-intersections.
\newblock {\em Topology}, 40(1):157--196, 2001.

\bibitem{ESz}
Tobias Ekholm and Andr{\'a}s Sz{\H{u}}cs.
\newblock Geometric formulas for {S}male invariants of codimension two
  immersions.
\newblock {\em Topology}, 42(1):171--196, 2003.

\bibitem{Esz2}
Tobias Ekholm and Andr{\'a}s Sz{\H{u}}cs.
\newblock The group of immersions of homotopy {$(4k-1)$}-spheres.
\newblock {\em Bull. London Math. Soc.}, 38(1):163--176, 2006.

\bibitem{fulton}
William Fulton.
\newblock {\em Algebraic curves}.
\newblock Universit{\'e} de Versailles, 2005.

\bibitem{hirsch}
Morris~W Hirsch.
\newblock Immersions of manifolds.
\newblock {\em Trans. AMS}, 93(2):242--276, 1959.

\bibitem{HM}
John~F Hughes and Paul~M Melvin.
\newblock The {S}male invariant of a knot.
\newblock {\em Comment. Math. Helvetici}, 60(1):615--627, 1985.

\bibitem{hughes}
John~Forbes Hughes.
\newblock Bordism and regular homotopy of low-dimensional immersions.
\newblock {\em Pacific J. Math}, 156(1):155--184, 1992.

\bibitem{husemoller}
Dale Husem{\"o}ller.
\newblock {\em Fibre bundles}, volume~20 of {\em Graduate Texts in
  Mathematics}.
\newblock Springer, 1994.

\bibitem{kinjo}
Shumi Kinjo.
\newblock Immersions of 3-sphere into 4-space associated with {D}ynkin diagrams
  of types {A} and {D}.
\newblock {\em arXiv:1309.6526}, 2013.

\bibitem{KirbyMelvin}
Rob Kirby and Paul Melvin.
\newblock Canonical framings for 3-manifolds.
\newblock {\em Turkish Journal of Mathematics}, 23(1):89--115, 1999.

\bibitem{looijenga}
Eduard Looijenga.
\newblock {\em Isolated singular points on complete intersections}, volume~77
  of {\em London Math. Society Lecture Note Series}.
\newblock Cambridge Univ. Press, 1984.

\bibitem{Mond1}
David Mond.
\newblock Singularities of mappings from surfaces to 3-spaces.
\newblock {\em Singularity Theory, D.T. L\^e, K. Saito, B. Teissier editors,
  World scientific}, pages 509--526.

\bibitem{Mond2}
David Mond.
\newblock Some remarks on the geometry and classification of germs of maps from
  surfaces to 3-space.
\newblock {\em Topology}, 26(3):361--383, 1987.

\bibitem{Mondwh}
David Mond.
\newblock The number of vanishing cycles for a quasihomogeneous mapping from
  $\bf{C}^2$ to $\bf{C}^3$.
\newblock {\em Quart. J. Math. Oxford}, 42:335--345, 1991.

\bibitem{mumford}
David Mumford.
\newblock The topology of normal singularities of an algebraic surface and a
  criterion for simplicity.
\newblock {\em Publ. Math. de l'IHES}, 9(1):5--22, 1961.

\bibitem{saeki}
Osamu Saeki, Andr{\'a}s Sz{\H{u}}cs, and Masamichi Takase.
\newblock Regular homotopy classes of immersions of 3-manifolds into 5-space.
\newblock {\em manuscripta math.}, 108(1):13--32, 2002.

\bibitem{smale}
Stephen Smale.
\newblock The classification of immersions of spheres in {E}uclidean spaces.
\newblock {\em Annals of Mathematics}, 69(2):327--344, 1959.

\bibitem{invariant}
Tonny~A. Springer.
\newblock {\em Invariant Theory}, volume 585 of {\em Lecture Notes in Math.}
\newblock Springer, 1977.

\bibitem{steenrod}
Norman~Earl Steenrod.
\newblock {\em The topology of fibre bundles}, volume~14 of {\em PMS}.
\newblock Princeton Univ. Press, 1951.

\bibitem{szucstwo}
Andr{\'a}s Sz{\H{u}}cs.
\newblock Two theorems of {R}okhlin.
\newblock {\em Journal of Mathematical Sciences}, 113(6):888--892, 2003.

\bibitem{Whitney}
Hassler Whitney.
\newblock On regular closed curves in the plane.
\newblock {\em Compositio Mathematica}, 4:276--284, 1937.

\end{thebibliography}
\end{document}